\documentclass[11pt, a4paper]{amsart}
\usepackage{amsthm, dsfont, a4}
\usepackage{hyperref}
\usepackage[all]{xy}
\usepackage[active]{srcltx}
\usepackage[short,nodayofweek]{datetime}
\usepackage{xcolor}
\usepackage{verbatim}

\newcommand\FF{{\mathbb F}}
\newcommand\GG{{\mathbb G}}
\newcommand\NN{{\mathbb N}}

\newcommand\CC{{\mathbb C}}
\newcommand\TT{{\mathbb T}}
\newcommand{\LL}{\mathbb{L}}
\newcommand{\Fp}{{\FF_{\! p}}}
\newcommand{\Fq}{{\FF_{\! q}}}
\newcommand{\Fpp}{{\FF_{\! \pfrak}}}
\newcommand{\FI}{{\FF_{\!\infty}}}
\newcommand{\CI}{{\CC_{\infty}}}
\newcommand{\cO}{\mathcal{O}}

\newcommand{\cR}{\mathcal{R}}
\newcommand{\cS}{\mathcal{S}}
\newcommand{\cF}{\mathcal{F}}

\newcommand{\sA}{\mathsf{A}}
\newcommand{\CAlg}{\mathsf{Alg}}
\newcommand{\CGrp}{\mathsf{Grp}}

\newcommand\afrak{{\mathfrak a}}
\newcommand\bfrak{{\mathfrak b}}
\newcommand\pfrak{{\mathfrak p}}
\newcommand\qfrak{{\mathfrak q}}
\newcommand{\up}[1]{{u_\pfrak^{\, #1}}}   
\newcommand{\degp}{{d_\pfrak}} 
\newcommand{\Epfrak}{E((\CA)_\pfrak)}
\newcommand{\jp}{\jmath_\pfrak}
\newcommand{\ip}{\iota_\pfrak}

\newcommand{\mot}{\mathsf{M}}  
\newcommand{\dumot}{\mathfrak{M}}  
\newcommand{\LA}{A\!\otimes\! L}
\newcommand{\CA}{A \!\otimes\! \CI}
\newcommand{\FA}{A \!\otimes\! F}
\newcommand{\LsepA}{A \!\otimes\! L^\sep}
\newcommand{\fsf}{\mathfrak{sf}}
\newcommand{\fd}{\mathfrak{d}}
\newcommand{\locLA}{\LA[F^{-1}]}
\newcommand{\locLAGtilde}{\LA[\tilde{G}^{-1}]}
\newcommand{\locLAG}{\LA[G^{-1}]}

\newcommand{\locLt}{L[t][F^{-1}]}
\newcommand{\GPsi}{\Gamma_{\!\Psi}}  

\newcommand{\one}{\mathds{1}}
\newcommand{\transp}[1]{#1^{\rm tr}}   
\newcommand{\dual}[1]{#1^\vee}

\newcommand{\bigmid}{\, \Big| \,}
\newcommand{\sep}{\mathrm{sep}}
\newcommand{\per}{\mathrm{per}}

\newcommand{\dphi}{d\phi}
\newcommand{\hdt}[2]{\partial_t^{(#1)}\!\!\left(#2\right)\!}
\newcommand{\hdte}[1]{\partial_t^{(#1)}}  
\newcommand{\hd}[3]{\partial_{#1}^{(#2)}\!\!\left(#3\right)\!}
\newcommand{\hde}[2]{\partial_{#1}^{(#2)}}  
\newcommand{\e}{\exp_E}  
\newcommand{\id}{\mathrm{id}}
\newcommand{\res}{\mathrm{res}}
\newcommand{\ev}{\mathrm{ev}}
\newcommand{\tr}{\mathrm{tr}}
\newcommand{\rk}{\mathrm{rk}}
\newcommand{\vs}{\Fq-\mathrm{v.s.}}
\newcommand{\D}[1]{\mathcal{D}_{#1}} 
\newcommand{\Dup}{\D{\up{}}} 
\newcommand{\trk}[1]{#1\textrm{-}\mathrm{rk}}

\newcommand{\tate}{\CC_\infty\cs{t}}  
\newcommand{\ps}[1]{[\![#1]\!]}  
\newcommand{\ls}[1]{(\!(#1)\!)}
\newcommand{\cs}[1]{\langle #1 \rangle} 

\newcommand{\vect}[1]{\text{\boldmath $#1$\unboldmath}} 
\newcommand{\svect}[2]{\left( \begin{matrix} {#1}_{1}\\ \vdots \\ {#1}_{#2}\end{matrix}\right)}

\newcommand{\isom}{\cong}

\newcommand{\betr}[1]{\lvert #1\rvert}
\newcommand{\norm}[1]{\left\lVert #1\right\rVert}

\DeclareMathOperator{\Aut}{Aut}
\DeclareMathOperator{\Hom}{Hom}
\DeclareMathOperator{\End}{End}

\DeclareMathOperator{\Mat}{Mat}
\DeclareMathOperator{\GL}{GL}
\DeclareMathOperator{\Gal}{Gal}
\DeclareMathOperator{\Lie}{Lie}

\DeclareMathOperator{\Ker}{ker}
\DeclareMathOperator{\Coker}{coker}
\DeclareMathOperator{\Quot}{Quot}

\DeclareMathOperator{\Ext}{Ext}

\newcommand{\LieE}{\Lie E}

\theoremstyle{plain}
\newtheorem{thm}{Theorem}[section]
\newtheorem{cor}[thm]{Corollary}
\newtheorem{lem}[thm]{Lemma}
\newtheorem{prop}[thm]{Proposition}

\newtheorem{mainthm}{Theorem}

\theoremstyle{definition}
\newtheorem{defn}[thm]{Definition}
\newtheorem{exmp}[thm]{Example}
\newtheorem{rem}[thm]{Remark}

\begin{document}

\title[Comparison isomorphisms and Galois representations]{Non-abelian Anderson A-modules: Comparison isomorphisms and Galois representations}
\author{Andreas Maurischat}
\address{\rm {\bf Andreas Maurischat}, Lehrstuhl A f\"ur Mathematik, RWTH Aachen University, Germany }
\email{\sf andreas.maurischat@matha.rwth-aachen.de}

\newdate{date}{14}{08}{2024}
\date{\displaydate{date}}



\begin{abstract}
In this manuscript, we consider non-abelian Anderson $A$-modules $E$ (of generic characteristic).
The main results are on the structure of their motives, and on comparison isomorphisms between their cohomological realizations.
In the center of these comparison isomorphisms, there is the space of special functions $\fsf(E)$ as defined by Gazda and the author in \cite{qg-am:sfgtshrd}. We also provide a generalization of Anderson's result on the equivalence of uniformizability of the Anderson module and rigid analytic triviality of its associated motive.

We contribute results that are new even in the case of abelian Anderson modules. For every non-zero prime ideal $\pfrak$ of $A$, the relation of the space of special functions to the $\pfrak$-adic Tate module 
provides a way to obtain $\pfrak^{n+1}$-torsion as special values of hyperderivatives of these special functions.
Using this result for uniformizable Anderson modules, we are able to describe the $\pfrak$-adic Galois representation via a rigid analytic trivialization, and hence give a direct link between the image of the $\pfrak$-adic Galois representation and the motivic Galois group. This generalizes results of various authors.
\end{abstract}

\maketitle

\setcounter{tocdepth}{2}
\tableofcontents

\section{Introduction}

%
%
%

Abelian or $A$-finite Anderson $A$-modules are equipped with several cohomological realizations, e.g.~a Betti realization, $\pfrak$-adic realizations, and a de Rham realization (cf.~\cite[Section 5]{uh-akj:pthshcff}). The de Rham realization a priori depends on whether an Anderson module $E$ is abelian or $A$-finite, but 
both definitions coincide if $E$ is both abelian and $A$-finite, and we showed in \cite{am:aefam} that an Anderson module is abelian if and only if it is $A$-finite. Hence both definitions are indeed equivalent.

Another property called \emph{uniformizability} is also need for an abelian Anderson module such that the Betti and de Rham realization functors are indeed faithful. In that case, all the cohomological realizations are locally free of the same rank, and one has comparison isomorphisms between these realizations over suitable extensions (see \cite[Section 5.7]{uh-akj:pthshcff}).

If the Anderson module is abelian and uniformizable, all these realizations can also be obtained via realizations of its corresponding $A$-motive (see \cite[Lemma 5.44 ff.]{uh-akj:pthshcff}). However, in the proof for obtaining the de Rham cohomology via the $A$-motive (see \cite{eug:drcdm} for the case of Drinfeld modules, and \cite{db-mp:ligvpc} for the general case), Anderson modules enter the picture which are not abelian, but still uniformizable.
These Anderson modules, called quasi-periodic extensions, were also central in the investigation of Gamma values in \cite{db-mp:ligvpc}.

Since non-abelian Anderson modules play this important role, this raises the question which of the properties do transfer to Anderson modules which are not abelian, but still uniformizable, or which are neither of both.

One part of this manuscript is devoted to the study of the structure of $A$-motives and dual $A$-motives of non-abelian Anderson modules. For the $A$-motives, we even provide an analog of Anderson's theorem on the relation between uniformizability and rigid analytic triviality.

We also analyze the cohomological realizations and the comparison isomorphisms for non-abelian Anderson modules.

\medskip

A further aim of this manuscript is to generalize to uniformizable Anderson modules the results in \cite{am-rp:tcagfdte} on the description of their torsion points via Taylor coefficients of special functions as well as the description of their Galois representations via Taylor expansions of the entries of rigid analytic trivializations.

\medskip

Throughout this manuscript, we deal with Anderson modules of generic characteristic, although certain properties might be valid also for Anderson modules of special characteristics.

\subsection{Cohomological realizations and comparison isomorphisms}

Although the realization functors might be not faithful for non-abelian (i.e.~non-$A$-finite) Anderson $A$-modules, we show that
the crucial property for all the comparison isomorphisms to exist is that the Anderson module is uniformizable. 
For the Betti homology and de Rham homology, this comparison isomorphism is already given in \cite{qg-am:sfgtshrd} for all Anderson modules, even non-abelian and non-uniformizable ones. In this article, we also bring into the picture the $\pfrak$-adic realizations for arbitrary primes~$\pfrak$.

In the heart of these isomorphisms is the space of special functions $\mathfrak{sf}(E)$ of the Anderson module $E$, as introduced in \cite{qg-am:sfgtshrd}. For an Anderson module $(E,\phi)$ over some field $L\subset \CI$, it is defined by
\[  \fsf(E) :=\{ h\in E(\TT) \mid \forall a\in A: \phi_a (h)=a\cdot h\}, \]
where $E(\TT)$ denote the points of $E$ in the Tate-algebra $\TT = A \hat{\otimes} \CI$, $\phi_a$ denotes the $A$-action on $E$, and $a\cdot h$ the usual multiplication in $\TT$. (For precise definitions see Section \ref{sec:notation}.)

Let $u\in A$ be such that $A/\Fq[u]$ is a finite separable extension, then one has two isomorphisms (see \cite[Theorems 3.11 and 3.3]{qg-am:sfgtshrd})
\[  \mathfrak{d}_{A/\Fq[u]}\otimes_A \Lambda_E \xrightarrow{\delta_u} \fsf(E) \xrightarrow{\cong} \Hom_{\TT}^\tau\left(\mot(E)_{\TT},\TT\right) \] 
where $\mathfrak{d}_{A/\Fq[u]}$ denotes the different ideal of the extension which is isomorphic to the dual of the space of differentials $\Omega_{A/\Fq}^1$ (cf.~\cite[Remark 3.14]{qg-am:sfgtshrd}), and $\mot(E)_\TT$ is the $A$-motive with coefficients extended to $\TT$. In other words, the comparison isomorphism from the Betti homology $\Lambda_E$ to the de Rham homology
$\Hom_{\TT}^\tau\left(\mot(E)_{\TT},\Omega_{A/\Fq}^1\otimes_A \TT \right)$ factors through the space of special functions, or -- more precisely -- through
$\Omega_{A/\Fq}^1\otimes_A \fsf(E)$.

For stating the link to the $\pfrak$-adic realizations, we let $A_\pfrak$ denote the complete local ring of $A$ at the prime ideal $\pfrak$, $\Fpp=A/\pfrak$ the residue field, and $\hat{\Omega}_{A_\pfrak/\Fpp}$ the space of continuous differential $1$-forms of $A_\pfrak$ (which is isomorphic to $A_\pfrak d\up{}$ for any uniformizer $\up{}$ at $\pfrak$).

\begin{mainthm}\label{mainthm:sf-in-tate-module} (Theorem \ref{thm:isom-jp}, Proposition \ref{prop:H-direct-summand} and Corollary \ref{cor:uniformizable-implies-constant-prank})\\
Let $(E,\phi)$ be a (possibly non-abelian) Anderson $A$-module of generic characteristic, $\pfrak \lhd A$ a non-zero prime ideal of $A$, and $\fsf(E)$ the space of special functions of $E$. Then there is a natural injective homomorphism of $A_\pfrak$-modules
\[ \hat{\Omega}_{A_\pfrak/\Fpp}\otimes_A \fsf(E) \longrightarrow T_\pfrak(E).\]
The image of $\hat{\Omega}_{A_\pfrak/\Fpp}\otimes_A \fsf(E)$ in $T_\pfrak(E)$ is always a direct summand.
If $E$ is uniformizable, this homomorphism is an isomorphism.
\end{mainthm}

Hence, the space of special functions is in the center of these three cohomological realizations. 

\medskip 

In \cite{am:ptmsv}, we already considered this situation in the case of $A=\Fq[t]$ (the space of special functions was denoted $H_E$ there), and established the diagram
\[ \xymatrix{
 & \{ (e_i)_{i\geq 0}\in T_t(E) \mid \lim_{i \to \infty}\norm{e_i}=0\} 
 \ar@{<-}[d]^{\cong } &  \\
 \Lambda \ar[r]^{\delta}_{\cong } &  
 \fsf(E) \ar[r]^(.4){\iota}_(.4){\cong }
 & \Hom_{\TT}^\tau(\mot_\TT, \TT) \\
&  (\dumot \otimes_{\CI[t]} \TT)^\sigma \ar@{<-}[u]_(.4){\cong } & 
}\]
which is nothing else than the isomorphisms stated above for $A=\Fq[t]$ and $\pfrak=(t)$, since $\Omega_{\Fq[t]/\Fq}^1=\Fq[t]dt\cong \Fq[t]$. 
Also the special cases of Theorem \ref{mainthm:sf-in-tate-module} are shown there, i.e.~that $\fsf(E)$ naturally injects into the $t$-adic Tate module $T_t(E)$ of $E$ as a direct summand, and that, if moreover $E$ is uniformizable, the space $\fsf(E)$ even generates the $t$-adic Tate module.

Since we focus on the $A$-motive later on, we didn't search for a generalization of the isomorphism $\fsf(E)\to (\dumot \otimes_{\CI[t]} \TT)^\sigma$ involving the dual $A$-motive $\dumot$.

\medskip

Theorem \ref{mainthm:sf-in-tate-module} has nice consequences. Namely, let 
\[\trk{\pfrak}(E)=\dim_\Fpp(E[\pfrak]) = \rk_{A_\pfrak}(T_\pfrak(E)) \]
be the $\pfrak$-rank of $E$ for any nonzero prime $\pfrak$ of $A$.
Then Theorem \ref{mainthm:sf-in-tate-module} implies that for uniformizable Anderson modules $E$, the $\pfrak$-rank is independent of the prime $\pfrak$ (see Corollary \ref{cor:uniformizable-implies-constant-prank}), and furthermore that 
uniformizable Anderson modules are regular in the sense of Yu (see Corollary~\ref{cor:uniformizable-implies-regular}).

\subsection{$A$-motives of non-abelian Anderson $A$-modules}

In the abelian case, the $A$-motive $\mot(E)$ of an Anderson $A$-module over $L$ is -- by definition -- finitely generated locally free as $\LA$-module. In the non-abelian case, however, $\mot(E)$ is not finitely generated and might have torsion.
However, we show that after some localization, the motive becomes finitely generated free. More precisely, we prove the following theorem.

\begin{mainthm}\label{mainthm:A-motive} (Theorem \ref{thm:rational-A-motive-f-dim})\\ 
Denote by $\Quot(\LA)$ the field of fractions of $\LA$.
The $A$-motive $\mot:=\mot(E)$ associated to $E$ has the following properties:
\begin{enumerate}
\item $\Quot(\LA)\otimes_{\LA}\mot $ is a finite dimensional $\Quot(\LA)$-vector space. 
\item There is an element $f\in \LA$ such that one has:
Let $F$ be the multiplicative subset generated by $f$ and all its images under iterated $\tau$-twists, then $(\locLA)\otimes_{\LA} \mot$ is finitely generated as $(\locLA)$-module.
\item The element $f\in \LA$ and the corresponding set $F$ as in \eqref{item:2}, can be chosen such that $(\locLA)\otimes_{\LA} \mot$ is a free module.
\end{enumerate}
\end{mainthm}

This implies that the $A$-motives of non-abelian Anderson modules don't behave too bad, but that geometrically one has to exclude only a small set.
The analogous statement for dual $A$-motives is also valid (see Theorem \ref{thm:rational-dual-A-motive-f-dim}).

Due to Theorem \ref{mainthm:A-motive}, $\Quot(\LA)\otimes_{\LA}\mot $ has finite dimension, and we call that dimension the rank of the motive, $\rk(\mot)$. If $E$ is abelian, Anderson showed that for all primes $\pfrak$ of $A$, the $\pfrak$-rank defined above equals the rank $\rk(\mot)$ of the $A$-motive (see \cite[Proposition 1.8.3]{ga:tm}). In the non-abelian case, we show

\begin{mainthm}(Theorem \ref{thm:equality-of-ranks})\\
Let $E$ be an Anderson $A$-module over $L$ and $\mot$ its $A$-motive. Then for almost all primes $\pfrak$ of $A$, we have
\[ \trk{\pfrak}(E)=\rk(\mot). \]
If $E$ is uniformizable, this equality holds for all primes $\pfrak$ of $A$.
\end{mainthm}

\subsection{Taylor series expansions and Taylor coefficients}

The homomorphism in Theorem \ref{mainthm:sf-in-tate-module} can be expressed using Taylor series expansions, and we use those in the uniformizable case to give a description of $\pfrak$-power torsion in terms of values of hyperderivatives of special functions (see Section \ref{sec:torsion-points-as-values}). This is a direct generalization of the case of Drinfeld modules given in \cite[Proposition~3.1 \& Theorem~3.5]{am-rp:tcagfdte}.

So as before, let $\pfrak$ be a prime ideal of $A$, and $A_\pfrak$ the $\pfrak$-adic completion of $A$. Denote by $(\CA)_\pfrak$ the $\pfrak$-adic completion
\[ (\CA)_\pfrak:= \varprojlim_{n} \left( (A/\pfrak^{n+1})\otimes\CI\right). \]

In Section \ref{sec:taylor-series-expansion}, we introduce the hyperdifferential operators $\hdte{n}$ with respect to elements $t\in A$, 
as well as for $h\in (\CA)_\pfrak$, the corresponding Taylor series expansion at $\pfrak$ which is given by
\[ \D{t}(h)= \sum_{n=0}^\infty \hdt{n}{h}(\pfrak) X^n\in (\Fpp\otimes \CI)\ps{X}\]
having as Taylor coefficients the evaluations at $\pfrak$ of the hyperderivatives $\hdt{n}{h}$. (See Section \ref{sec:taylor-series-expansion} for definition and details.)

\begin{mainthm}(Theorem \ref{thm:torsion-as-special-values})\\
Let $(E,\phi)$ be a uniformizable Anderson $A$-module, and let $h_1,\ldots, h_r\in \fsf(E)$ be an $A_{\pfrak}$-basis of $A_{\pfrak}\otimes_A \fsf(E)\cong T_\pfrak(E)$.
Further, let $t\in A\setminus \Fq$ be such that $\partial_{\up{}}^{(1)}(t)\in A_{\pfrak}^\times$ where $\up{}$ is a uniformizer at $\pfrak$. 
Then
\begin{enumerate}
\item The hyperdifferential operators $(\hdte{n})_{n\geq 0}$ on $\Fq[t]$ uniquely extend to $A_{\pfrak}$.
\item For all $h\in A_{\pfrak}\otimes_A \fsf(E)$ and $n\geq 0$,  one has $\partial_{t}^{(n)}(h)(\pfrak)\in \Fpp\otimes E[\pfrak^{n+1}]$ as well as
\[    \partial_{t}^{(n)}(h)(\pfrak) \equiv \left( \partial_{\up{}}^{(1)}(t)^{-n}\cdot \partial_{\up{}}^{(n)}(h)\right) (\pfrak) \mod \Fpp\otimes E[\pfrak^{n}]. \]
\item Let $\tr: \Fpp\to \Fq$ be the trace map.
For all $n\geq 0$, the elements $\{ (\tr\otimes\id)\left(\partial_{t}^{(n)}(h_i)(\pfrak)\right) \mid i=1,\ldots, r\}$ provide an $(A/\pfrak^{n+1})$-basis of $E[\pfrak^{n+1}]$. 
\end{enumerate}
\end{mainthm}

As a special case by taking $n=0$, we recover the Gauss-Thakur sums $h(\pfrak)$ considered in \cite{qg-am:sfgtshrd}.

Part two and three tell us several things. First, we obtain all $\pfrak^{n+1}$ torsion from the Taylor coefficients $\hdt{j}{h_i}(\pfrak)$ for $0\leq j\leq n$ and $1\leq i\leq r$, and second that this holds for any chosen element $t\in A$, as long as the derivative $\partial_{\up{}}^{(1)}(t)$ is invertible in $A_{\pfrak}^\times$. Third, if we change $t$, the basis of $E[\pfrak^{n+1}]/E[\pfrak^{n}]$ is only changed by a scalar, which is moreover explicit.

\subsection{Galois representations}

In \cite{am-rp:tcagfdte}, the investigations on the Taylor coefficients led to a description  of the $\pfrak$-adic Galois representation in terms of a rigid analytic trivialization of the corresponding $t$-motive, and to a direct proof that the image of the Galois representation lies in the motivic Galois group. This was possible, as in the case of Drinfeld modules such a rigid analytic trivialization is given by Anderson generating functions and their Frobenius twists.

As mentioned above, such a link is also present in the general case of uniformizable Anderson modules $E$ via the isomorphism 
$\fsf(E) \xrightarrow{\cong} \Hom_{\TT}^\tau\left(\mot(E)_{\TT},\TT\right)$
 (see \cite[Theorem~3.3]{qg-am:sfgtshrd}). Combining these two links of the space of special functions $\fsf(E)$, we obtain the Galois representation in terms of a rigid analytic trivialization also in the general case of a (possibly non-abelian) uniformizable Anderson module.
Special cases of such a description are also present in \cite{cc-mp:aipldm}, \cite{am:atmttgr}.

As above, we let $\pfrak$ be a prime ideal of $A$, and let $(\CA)_\pfrak$ be the $\pfrak$-adic completion
\[ (\CA)_\pfrak:= \varprojlim_{n} \left( (A/\pfrak^{n+1})\otimes\CI\right) \]
as well as by $\TT_{(\pfrak)}\subseteq \Quot(\TT) $ the localization 
of $\TT$ away from the divisors of $\pfrak$.

We remark that $\TT_{(\pfrak)}$ naturally embeds into $(\CA)_\pfrak$ (see Proposition~\ref{prop:inclusion-of-tt_p-in-padic}).

\begin{mainthm}\label{mainthm:Galois-representation}(see Theorem \ref{thm:galois-representation} and Theorem \ref{thm:image-in-motivic-group})\\
Assume that $E$ is uniformizable, and that the localization $\mot(E)_{(\pfrak)}=(\LA)_{(\pfrak)}\otimes_{\LA} \mot$ of the $A$-motive at the prime $\pfrak$ is free and finitely generated.
Let $\Upsilon\in \GL_r( \TT_{(\pfrak)})$ be a rigid analytic trivialization of $\mot(E)_{(\pfrak)}$ with respect to some $(\LA)_{(\pfrak)}$-basis of $\mot_{(\pfrak)}$. 
Using the embedding $\TT_{(\pfrak)}\to (\CA)_\pfrak$, we consider $\Upsilon\in \GL_r((\CA)_\pfrak)$.
\begin{enumerate}
\item The coefficients of $\Upsilon$ lie in $(\LsepA)_\pfrak$.
\item The $\pfrak$-adic Galois representation attached to $E$ (via the Tate module) is given by
\begin{eqnarray*}
\varrho_\pfrak: \Gal(L^{\sep}/L) &\to& \GL_r(A_\pfrak) \\
\gamma &\mapsto & \Upsilon\cdot \gamma(\Upsilon)^{-1}.
\end{eqnarray*}
with respect to an appropriate choice of basis. 
\item The image of the Galois representation $\varrho_\pfrak$ lies in $\Gamma_\mot(K_\pfrak)$, where $\Gamma_\mot$ denotes the motivic Galois group of $\mot$, and $K_\pfrak=\Quot(A_\pfrak)$.
\end{enumerate}
\end{mainthm}

\section{Generalities}\label{sec:notation}

\subsection{Basic objects}

Let $\Fq$ be a finite field with $q$ elements and characteristic $p$, and let $C$ be a smooth projective geometrically irreducible curve over $\Fq$. We fix a closed point $\infty$ on $C$ and consider $A=H^0(C\setminus \{\infty\},\cO_C)$, the ring of rational functions on $C$ that are regular outside $\infty$. This is an algebra of Krull dimension $1$ over the field~$\Fq$. 

Let $K$ be the function field of $C$ (or equivalently, the fraction field of $A$). The degree and residue field of $\infty$ will be denoted by $d_{\infty}$ and $\FI$, respectively, and the associated norm will be $|\cdot|$. Further, let $\CI$ be the completion at a place above $\infty$ of an algebraic closure of $K$, to which we extend $|\cdot|$. 
By convention, every unlabeled tensor will be over $\Fq$. 

For an ideal $(0)\ne \afrak\lhd A$, we write $\deg(\afrak):=\dim_{\Fq}(A/\afrak)$, and accordingly for $a\in A\setminus\{0\}$, $\deg(a):=\dim_{\Fq}(A/(a))$.

For a prime ideal $\pfrak\lhd A$,\footnote{When we speak of a prime ideal of $A$, we always mean a prime ideal different from $(0)$. At some places, we will nevertheless emphasize that the prime ideal is non-zero.} the $\pfrak$-adic completion is denoted by
$A_\pfrak= \varprojlim_n A/\pfrak^n$, and $K_\pfrak$ is the field of fractions of $A_\pfrak$.
Further, $\Fpp= A/\pfrak=A_\pfrak/(\pfrak)$ denotes the residue field.\\
As it is well-known, there is a unique homomorphism of $\Fq$-algebras $\Fpp\hookrightarrow A_\pfrak$ such that the composition with the projection map $A_\pfrak\to \Fpp$ is the identity on $\Fpp$, and for any uniformizer $\up{}\in A$ of the prime ideal $\pfrak$, the completion $A_\pfrak$ is isomorphic to the power series ring $\Fpp\ps{\up{}}$, and $K_\pfrak$ is isomorphic to the Laurent series ring $\Fpp\ls{\up{}}$.

In the whole paper, for each prime ideal $\pfrak$ of $A$, we will fix a uniformizer $\up{}$, and the isomorphism $\Fpp\ps{\up{}}\isom A_\pfrak$ which induces the isomorphisms $\Fpp\ls{\up{}}\isom K_\pfrak$ as well as $ \Fpp\ps{\up{}}/(\up{n})\isom A_\pfrak/(\pfrak^n)=A/\pfrak^n$. 
In this way, we consider $A_\pfrak$, $K_\pfrak$ and the quotients $A/\pfrak^n=A_\pfrak/(\pfrak^n)$
 as $\Fpp$-algebras. Be aware that $A$ is not an $\Fpp$-algebra unless $\Fpp=\Fq$.
 
\medskip

We fix $L$ an intermediate field $K\subseteq L\subseteq \CI$, and denote the natural inclusion by $\ell:K\to L$. The separable algebraic closure of $L$ inside $\CI$ will be denoted by $L^\sep$.

Let $(E,\phi)$ be an Anderson $A$-module over $L$ of dimension $d$ (and generic characteristic) as in \cite[Definition 5.2]{uh-akj:pthshcff}, i.e.~$E\cong \GG_{a,L}^d$ as $\Fq$-vector space schemes and
$\phi:A\to \End_{\vs}(E)$ is a ring homomorphism, such that for all $a\in A$: $\dphi_a-\ell(a):\ \LieE\to \LieE$ is nilpotent. Here, $\LieE$ denotes the usual Lie algebra of the group scheme $E$ over $\CI$ (which is a $\CI$-vector space of dimension $d$), and 
$\dphi_a\in \End_{\CI}(\LieE)$ denotes the endomorphism on the Lie algebra of $E$ induced by $\phi_a:=\phi(a)$.

Attached to an Anderson $A$-module $E$ is its exponential map 
$\e:\LieE\to E(\CI)$, which is the unique $\Fq$-linear map with the properties 
\[\forall a\in A, x\in\LieE:\ \phi_a(\e(x))=\e(\dphi_a(x)) \]
and $d(\e)=\id_{\LieE}$.

The kernel of the exponential map is called the \emph{period lattice}, denoted by $\Lambda_E$. By the functional equation,
the period lattice is an $A$-module (using the action via $\dphi$). Further, it is well-known that the exponential map is a local isometry, and hence $\Lambda_E$ is discrete in $\LieE$ (see e.g.~\cite[Lemma 5.4]{uh-akj:pthshcff}).

The Anderson module $E$ is called \emph{uniformizable} if the exponential map $\e$ is surjective.

\begin{rem}\label{rem:torsion-in-Lsep}
For all $a\in A\setminus \{0\}$, we have $\ell(a)\ne 0$, so $\dphi_a$ is an isomorphism. Hence $\phi_a:E(\CC_\infty)\to E(\CC_\infty)$ has finite kernel \[ E[a]:=\ker(\phi_a), \]
called the \emph{$a$-torsion} of $E$.
More general for an ideal $(0)\neq \afrak\lhd A$, we define
\[  E[\afrak]:= \bigcap_{a\in \afrak} E[a]. \]
This is an $A/\afrak$-module (via $\phi$). 
Since $\phi_a$ is a finite morphism defined over $L$ and $\dphi_a$ is an isomorphism, the torsion points have coefficients in $L^\sep$ i.e.~$E[\afrak]\subseteq E(L^\sep)$.
\end{rem}

For prime ideals $\pfrak$, we consider the $\pfrak$-adic Tate-module
\[ T_\pfrak(E):=\Hom_A(K_\pfrak/A_\pfrak, E(\CC_\infty)).\]
It is easy to see that for $f\in T_\pfrak(E)$, the image of each $x\in K_\pfrak/A_\pfrak$ is in $E[\pfrak^{n}]\subseteq E(L^\sep)$ for some $n\geq 0$. Hence, indeed
\[ T_\pfrak(E)=\Hom_A(K_\pfrak/A_\pfrak, E(\CC_\infty))=\Hom_A(K_\pfrak/A_\pfrak, E(L^\sep)).\]

The Tate-module can also be described as a projective limit $\varprojlim_{n} E[\pfrak^{n+1}]$
(cf.~\cite[Sect.~5.7]{uh-akj:pthshcff}).\\
As $\up{}\in A$ is a uniformizer, one obtains maps $\phi_{\up{}}:E[\pfrak^{n+1}]\to  E[\pfrak^n] $ for all $n\geq 0$ and the projective limit $\varprojlim_{n} E[\pfrak^{n+1}]$ with respect to these maps.
Then it is not hard to see, that the map
\[  T_\pfrak(E):=\Hom_A(K_\pfrak/A_\pfrak, E(\CC_\infty))\isom\Hom_A(\Fpp\ls{\up{}}/\Fpp\ps{\up{}}, E(\CC_\infty)) \longrightarrow \varprojlim_{n} E[\pfrak^{n+1}] \]
sending $f\in T_\pfrak(E)$ to $(f(\up{-n-1}))_{n\geq 0}$ is the desired isomorphism.
This isomorphism implies that if $\{f_1,\ldots, f_r\}$ is an $A_\pfrak$-basis of $T_\pfrak(E)$, then for each $n\geq 0$, the set of values $\{f_1(\up{-n-1}),\ldots, f_r(\up{-n-1})\}$ is an $A/\pfrak^{n+1}$-basis of $E[\pfrak^{n+1}]$.

\medskip

On the ring $\LA$, we consider the \emph{Frobenius twist} $\tau:\LA \to \LA$ given by $\tau(a\otimes l)=a\otimes l^q$, i.e. being the identity on $A$ and the Frobenius homomorphism on $L$. 

For an Anderson $A$-module $E$ over $L$, one considers the $(\LA)$-module
\[ \mot(E):=\Hom_{\vs}(E,\GG_{a,L}) \]
where the $A$-action is induced by the $A$-action on $E$, and the $L$-action is induced by scalar multiplication on $\GG_{a,L}$.
In addition, there is a $\tau$-semilinear endomorphism on $\mot(E)$ induced by the $q$-power Frobenius endomorphism on $\GG_{a,L}$.

As $E$ is a vector space scheme isomorphic to $\GG_{a,L}^d$, the module $\mot(E)$ is always a free and finitely generated module over 
$L\{\tau\}=\End_{\vs}(\GG_{a,L})$ of dimension $d$.

The Anderson module $E$ is called \emph{abelian} if $\mot(E)$ is also finitely generated as $\LA$-module. 
As in the abelian case, we will call this module \emph{the $A$-motive of $E$}, although in the non-abelian case, the functor $E\mapsto \mot(E)$ does not share the nice properties that one requires for a category of motives.

\subsection{Completions} \label{sec:completions}

In the ``$A=\Fq[t]$-setting'', we had to deal with the Tate-algebra $\tate$ and the power series ring
$\CI\ps{t}$ (see \cite{am:ptmsv}). The first is the completion of the polynomial ring $\CI[t]$ with respect to the Gauss norm, the second is the completion of $\CI[t]$ with respect to the ideal~$(t)$.

In the following, we will introduce the more general notions needed later.

\begin{itemize}
\item Given an $\Fq$-basis $B(A)$ of $A$, we can define a norm on $\CA$ via
\[   \norm{\sum_{a\in B(A)}^{\rm finite} a\otimes c_a\, } := \max \{  \betr{c_a} \mid a\in B(A) \}, \]
which is well-defined as every such sum has only finitely many $c_a$ different from $0$. 
A short computation shows  (see \cite[Proposition~2.2]{qg-am:sfgtshrd}) that this norm is independent of the chosen $\Fq$-basis of $A$.
We define the completion of $\CA$ with respect to this norm
\[ \TT:=A \widehat{\otimes} \CI
: =\left\{ \sum_{a\in B(A)} a\otimes c_a \mid c_a\in \CI,\, \lim_{\deg(a)\to \infty} \betr{c_a}=0 \right\},\footnote{Be aware that this definition of $\TT$ is indeed the usual completion of $\CA$ with respect to the given norm, since $\CI$ is a complete ring. If one replaces $\CI$ by a more general normed ring $F$, one would have to replace $F$ by its completion in the explicit description.} \]
which is also a normed ring by
\[   \norm{\sum_{a\in B(A)} a\otimes c_a\, } := \max \{  \betr{c_a} \mid a\in B(A) \}. \]
As the curve $C$ was assumed to be geometrically irreducible, the rings $\CA$ and $\TT$ are integral domains.
\item For all $0\neq \pfrak\lhd A$ prime ideals, and intermediate fields $K\subseteq F\subseteq \CI$, we write
\[  (\FA)_{\pfrak} := \varprojlim_{n} \left( (A/\pfrak^{n+1})\otimes F\right). \]
We will mainly use $(\CA)_{\pfrak}$, but also $(A\otimes L^\sep)_\pfrak$.

\end{itemize}

\begin{rem}\label{rem:notation-in-completion}
\begin{enumerate}
\item As every intermediate field $K\subseteq F\subseteq \CI$ is an infinite dimensional $\Fq$-vector space, the ring $(\FA)_\pfrak$ is strictly larger than $A_\pfrak\otimes F$.
\item Every element of $(\FA)_\pfrak$ can be written as a series $\sum_{k=0}^\infty \up{k}x_k =\lim\limits_{n\to \infty}\sum_{k=0}^n \up{k}x_k $ with $x_k\in \Fpp \otimes F$, since $A/\pfrak^{n+1}$ is isomorphic to $\Fpp[\up{}]/(\up{n+1})$
\item If $\deg(\pfrak)>1$, the rings $A_\pfrak\otimes \CI$ and $(\CA)_\pfrak$ have zero-divisors, as they contain a copy of $\Fpp\otimes \Fpp$.
\end{enumerate}
\end{rem}

\begin{prop}\label{prop:inclusion-of-tt-in-padic} \ 
For all primes $0\neq \pfrak\lhd A$, there is a natural inclusion of rings
\[  \TT \hookrightarrow (\CA)_{\pfrak}. \]
\end{prop}

\begin{proof}
For all $n\geq 0$, the natural projection $\pi_n:\CA\to (A/\pfrak^{n+1})\otimes \CI$ is continuous.
Furthermore,  as $(A/\pfrak^{n+1})$ is a finite dimensional $\Fq$-vector space,
$(A/\pfrak^{n+1})\otimes \CI$ is a finite dimensional $\CI$-vector space, and hence complete. Therefore, the natural projection $\pi_n$ can be extended continuously to a homomorphism
\[  \hat{\pi}_n:\TT=A \hat{\otimes} \CI\to (A/\pfrak^{n+1})\otimes \CI. \]
As the continuous extension is unique, it is clear that the $\hat{\pi}_n$ are compatible with the projections
$\pi_{n+1,n}:(A/\pfrak^{n+2})\otimes \CI\to (A/\pfrak^{n+1})\otimes \CI$. Hence, they induce a homomorphism to the projective limit
\[  \hat{\pi}: \TT \to (\CA)_\pfrak. \]
As the $\pi_n$ are homomorphisms of rings, so are the $\hat{\pi}_n$, and also $\hat{\pi}$ is a homomorphism of rings.
We claim that this homomorphism is injective.\\
Let $0\ne x=\sum_{a\in B(A)} a\otimes c_a\in A \hat{\otimes} \CI$ where as before $B(A)$ denotes an arbitrary $\Fq$-basis of $A$. By definition, we have
$\lim_{\deg(a)\to \infty} \betr{c_a}=0$, hence the set
\[  J:=\{ a\in B(A)\mid \betr{c_a}=\norm{x} \}\]
is finite. Let $x_0:=\sum_{a\in J} a\otimes c_a\in \CA$ and 
\[ x_1:=x-x_0=\sum_{a\in B(A)\setminus J} a\otimes c_a\in A \hat{\otimes} \CI, \]
then $\norm{x_1}=\max\{ \betr{c_a} \mid a\in B(A)\setminus J\} <\norm{x}=\norm{x_0}$.\\
Now choose $n\geq 0$ large enough such that the residue classes of $J$ in $A/\pfrak^{n+1}$ are still $\Fq$-linearly independent. Therefore, for such $n$, we have $\norm{\pi_n(x_0)}=\norm{x_0}$. On the other hand, $\norm{\hat{\pi}_n(x_1)}\leq \norm{x_1}<\norm{x_0}$. Hence,
\[ \norm{\hat{\pi}_n(x)}=\norm{\hat{\pi}_n(x_0)+\hat{\pi}_n(x_1)}=\max\{\norm{\pi_n(x_0)},  \norm{\hat{\pi}_n(x_1)} \} = \norm{x_0}\neq 0. \]
Therefore, $x\notin \Ker(\hat{\pi}_n)$, and $x\notin \Ker(\hat{\pi})$. As $x\neq 0$ was chosen arbitrarily, this implies that $\Ker(\hat{\pi})=0$, i.e.~$\hat{\pi}$ is injective.
\end{proof}

\begin{rem}\label{rem:TT-in-padic}
We can obtain this embedding quite explicit.
Namely, choose an $\Fq$-basis $B(A)$ of $A$ by first choosing lifts of a basis of $A/\pfrak$, and extending this set successively by lifts of bases of $\pfrak/\pfrak^2$, $\pfrak^2/\pfrak^3$ etc.
Then the elements in $(\CA)_\pfrak$ are just given by the formal infinite sums
$\sum_{a\in B(A)} a\otimes c_a$ for arbitrary $c_a\in \CI$, whereas the elements in
$\TT=A\hat{\otimes} \CI$ are given by those sums
$\sum_{a\in B(A)} a\otimes c_a$ for which the $c_a$ tend to zero.

Be aware that for $\deg(\pfrak)>1$, the $\up{}$-expansion in Remark \ref{rem:notation-in-completion} is not of this form, since $\Fpp\not\subset A$.
\end{rem}

For later use, we already give the definition of the evaluation at a prime $\pfrak$.

\begin{defn}\label{defn:evaluation}
For $h\in (\CA)_\pfrak$, we denote by $h(\pfrak)$ its image in $(\CA)_\pfrak/\pfrak(\CA)_\pfrak\isom \Fpp\otimes \CI$, and call it the \textbf{evaluation of $h$ at $\pfrak$}.
\end{defn}

\subsection{Localizations}

In this subsection, we introduce our notation for localizations at finite products of prime ideals.

\begin{defn}\label{def:localization}
For an integral domain $R$, let $\Quot(R)$ be its field of fractions. If $\afrak$ is an ideal in $R$, we set
\[  R_{(\afrak)} := \bigcap_{\pfrak\mid \afrak} R_{(\pfrak)} \subseteq \Quot(R) \]
to be the intersection of the localizations at all primes $\pfrak$ dividing $\afrak$.

More general, if $S\supseteq R$ is an integral domain containing $R$, and $\afrak$ an ideal in $R$, then $S_{(\afrak)}:=S_{(S\afrak)}$ denotes
the corresponding localization at the extension $S\afrak$ of the ideal $\afrak$.
\end{defn}

\begin{exmp}
The cases that we will have to deal with are the rings $(\LA)_{(\pfrak)}$ and $(\CA)_{(\pfrak)}$ for a prime ideal $\pfrak\subseteq A$,
as well as $\TT_{(\pfrak)}$.

We should also emphasize the difference between $(\CA)_{(\pfrak)}$ and $(\CA)_{\pfrak}$. The former being the localization of $\CA$ defined just now, and the latter being the completion defined in the previous paragraph.
\end{exmp}

\begin{prop}\label{prop:inclusion-of-tt_p-in-padic} \
For all primes $0\neq \pfrak\lhd A$, there is a natural inclusion of rings
\[  \TT_{(\pfrak)} \hookrightarrow (\CA)_{\pfrak}. \]
extending the embedding $\TT\to (\CA)_{\pfrak}$ given in Proposition \ref{prop:inclusion-of-tt-in-padic}.
\end{prop}

\begin{proof}
If $h\in \TT$ is not contained in any prime dividing $\pfrak\TT$, then its evaluation $h(\pfrak)\in \Fpp\otimes \CI$ is invertible. Hence, the image of $h$ in $(\CA)_\pfrak$ is invertible. Therefore the inclusion $\TT\hookrightarrow (\CA)_\pfrak$ can be extended to the localization $\TT_{(\pfrak)}$.
\end{proof}

For later use, we state the following lemma which is immediate from the previous discussion.

\begin{lem}\label{lem:invertible-elements}
For $h\in \TT_{(\pfrak)}$ the following are equivalent:
\begin{enumerate}
\item $h$ is invertible in $\TT_{(\pfrak)}$,
\item $h(\pfrak)\in \Fpp \otimes \CI$ is invertible,
\item $h$ is invertible in $(\CA)_\pfrak$.
\end{enumerate}
\end{lem}

\subsection{Hyperderivations and Taylor series expansions}\label{sec:taylor-series-expansion}

We give a short presentation of hyperderivations, also called higher derivations (see e.g.~\cite[\S 27]{hm:crt}, or \cite{kc:dp}).

\begin{defn}
A hyperderivation on an $\Fq$-algebra $R$ is a family of $\Fq$-linear operators $(\hde{}{n}:R\to R)_{n\geq 0}$ (called \emph{hyperdifferential operators}) such that
\begin{enumerate}
\item $\hde{}{0}=\id_R$,
\item $\forall n\geq 0$, $\forall f,g\in R$: $\hd{}{n}{f\cdot g} = \sum\limits_{i+j=n} \hd{}{i}{f}\cdot \hd{}{j}{g}$ \textit{(generalized Leibniz rule)}
\end{enumerate}
\end{defn}

\begin{rem}
By definition, the hyperdifferential operator $\partial:=\hde{}{1}$ is a derivation, and 
$\hd{}{n}{f}$ corresponds to $\tfrac{1}{n!}\partial^n(f)$ in characteristic zero.
\end{rem}

For many calculations, it is much more convenient to consider the corresponding map
\[ \D{}: R\to R\ps{X}, f\mapsto \sum_{n \geq 0} \hd{}{n}{f} X^n, \] called the \textit{(generic) Taylor series expansion}. Namely, the $\Fq$-linearity and the generalized Leibniz rule are equivalent to that the map $\D{}$ is a homomorphism of $\Fq$-algebras. In particular, $\D{}$ is determined by the images of generators of the $\Fq$-algebra $R$.

\begin{prop}\label{prop:unique-extensions} (see \cite[Theorem 27.2]{hm:crt} and \cite[Theorem~5 \& 6]{kc:dp})\\
Hyperderivations can be uniquely extended to \'etale ring extensions (in particular to
localizations and to separable algebraic field extensions).\\
If $R$ is equipped with an absolute value, and all $\hde{}{n}$ are continuous, there is a unique way to extend them continuously to a hyperderivation on the completion of $R$ with respect to the absolute value.
\end{prop}

The hyperderivations that are relevant in this paper are the 
hyperderivations with respect to some element in the coefficient ring $A$.

\begin{exmp}\label{exmp:hyperderivations-wrt-pfrak}
Let $\pfrak$ be a prime ideal of the coefficient ring $A$, and let $\up{}$ be a uniformizer at $\pfrak$. The hyperderivation $(\hde{\up{}}{n})_{n\geq 0}$ with respect to $\up{}$ is given on $\Fq[\up{}]$ by
\[  \hd{\up{}}{n}{\sum_{i=0}^m c_i\up{}^i} = \sum_{i=n}^m \binom{i}{n} c_i\up{}^{i-n} \]
for any $\sum_{i=0}^m c_i\up{}^i\in \Fq[\up{}]$ where $\binom{i}{n}$ is the binomial coefficient considered in $\Fp\subseteq \Fq$.

As $\up{}$ is a uniformizer at $\pfrak$, the (non-complete) localization $A_{(\pfrak)}$ is \'etale over $\Fq[\up{}]_{(\up{})}$, and hence by Proposition~\ref{prop:unique-extensions}, the hyperderivation $(\hde{\up{}}{n})_{n\geq 0}$ can be uniquely extended to a hyperderivation on $A_{(\pfrak)}$.
\end{exmp}

The next lemma shows that the hyperdifferential operators are also continuous with respect to the $\pfrak$-adic topology, and hence can further be extended to a hyperderivation on~$A_\pfrak$.

\begin{lem}\label{lem:hd-p-adic-continuous}
If $(\hde{}{n})_{n\geq 0}$ is a family of hyperdifferential operators on an $\Fq$-algebra $R$, and $\afrak$ an ideal of $R$, then for all $k\geq n\geq 0$: $\hd{}{n}{\afrak^k}\subseteq \afrak^{k-n}$.
\end{lem}

\begin{proof}
Let $a_1,\ldots, a_k\in \afrak$. From the generalized Leibniz rule we obtain
\[ \hd{}{n}{a_1\cdots a_k}=\sum_{i_1+\ldots+i_k=n} \hd{}{i_1}{a_1}\cdots  \hd{}{i_1}{a_k}. \]
As $\#\{ j\mid i_j=0\}\geq k-n$ for each tuple $(i_1,\ldots, i_k)$ of the sum, we obtain 
$\hd{}{n}{a_1\cdots a_k}\in \afrak^{k-n}$. By additivity, we obtain $\hd{}{n}{f}\in \afrak^{k-n}$ for any $f\in \afrak^k$.
\end{proof}

\begin{prop}
Via the isomorphism $A_\pfrak\isom \Fpp\ps{\up{}}$, the hyperdifferential operators $\hde{\up{}}{n}$ are given by
\[ \hd{\up{}}{n}{\sum_{j=0}^\infty c_j\up{j}} = \sum_{j=n}^\infty \binom{j}{n} c_j \up{j-n}.\]
\end{prop}

\begin{proof}
The hyperdifferential operators are trivial on $\Fq$, and $\Fpp$ is a finite separable extension of $\Fq$. Hence by Proposition~\ref{prop:unique-extensions}, they are trivial on $\Fpp$, too. The formula then follows by additivity and continuity of the hyperdifferential operators.
\end{proof}

For later use, we also extend the hyperdifferential operators $\hde{\up{}}{n}$ to $A_\pfrak\otimes \CI$ by $\CI$-linearity, and then continuously to $(\CA)_{\pfrak}= \varprojlim_{n} \left( (A/\pfrak^{n+1})\otimes\CI\right)$. Explicitly, for $B(A)$ as in Remark \ref{rem:TT-in-padic} and $\sum_{a\in B(A)}a\otimes c_a\in (\CA)_\pfrak$, we have
\[ \hd{\up{}}{n}{\sum_{a\in B(A)}a\otimes c_a} = \sum_{a\in B(A)} \hd{\up{}}{n}{a}\otimes c_a, \]
the right hand side being a series of elements in $A_\pfrak\otimes \CI$ which converges in $(\CA)_\pfrak$ by Lemma~\ref{lem:hd-p-adic-continuous}.

\begin{defn}\label{defn:taylor-series-at-up}
For $h\in (\CA)_\pfrak$, let $h(\pfrak)\in \Fpp\otimes \CI$ be its evaluation at $\pfrak$ as given in Definition \ref{defn:evaluation}.
Using this evaluation, we define the \emph{Taylor series expansion at $\pfrak$} (which depends on the uniformizer $\up{}$) to be the homomorphism of $\Fq$-algebras
\[\Dup:(\CA)_\pfrak \to (\Fpp\otimes \CI)\ps{X}, h\mapsto \sum_{n=0}^\infty \hd{\up{}}{n}{h}(\pfrak)X^n.\]
\end{defn}

\begin{prop}
The Taylor series expansion at $\pfrak$ is an isomorphism.
\end{prop}

\begin{proof}
Let $\alpha_1,\ldots, \alpha_\degp$ be an $\Fq$-basis of $\Fpp$. 
Then an arbitrary element in $h=(\CA)_\pfrak$ can uniquely be written as a ($\pfrak$-adically) convergent series
\[ h=\sum_{j=0}^\infty \sum_{i=1}^{\degp} \alpha_i\up{j}\otimes c_{ij} \]
with $c_{ij}\in \CI$. Then
\[ \left(\hd{\up{}}{n}{h}\right)(\pfrak)=\left(\sum_{j=n}^\infty \sum_{i=1}^{\degp} \binom{j}{n} \alpha_i\up{j-n}\otimes c_{ij}\right)(\pfrak) = \sum_{i=1}^{\degp} \alpha_i\otimes c_{in}. \]
Hence,
\[ \Dup\left(\sum_{j=0}^\infty \sum_{i=1}^{\degp} \alpha_i\up{j}\otimes c_{ij}\right)
= \sum_{n=0}^\infty \left(\sum_{i=1}^{\degp} \alpha_i\otimes c_{in}\right) X^n. \qedhere \]
\end{proof}

\begin{rem}\label{rem:up-expansion-by-Taylor-series}
We see from the computation in the previous proof that we obtain the  $\up{}$-expansion of an element $h\in (\CA)_\pfrak$ by
replacing $X$ by $\up{}$ in its Taylor series expansion at $\pfrak$, i.e. the $\up{}$-expansion of $h$ is given as
\[ h=\sum_{n\geq 0}\up{n}\cdot\hd{\up{}}{n}{h}(\pfrak). \]
\end{rem}

\subsection{Dualizable difference modules} \label{subsec:dualizable-difference-modules}

In this subsection, we briefly review some basics on difference rings and difference modules. As the usual definitions of difference modules are only given over fields or difference-pseudo fields (cf.~\cite{mvdp-mfs:gtde, am:dvnt, ao-mw:sgtlde}), we have to extend these definitions, and we introduce \emph{dualizable difference modules} which are the objects that are usually just called difference modules. The idea of considering these dualizable difference modules as a generalization was developed by M.~Wibmer and the author in an unpublished manuscript.

A ring $\cR$ with an endomorphism $\tau:\cR\to \cR$ is called a \emph{difference ring}.
A \emph{difference ring extension} $(\cS,\tau_\cS)$ of $(\cR,\tau)$ consists of a ring extension $\cS/\cR$ together with an endomorphism $\tau_\cS$ extending the endomorphims $\tau$. In such a case, we will usually omit the subscript $\cS$ and also write $\tau$ instead of $\tau_\cS$.

For a difference ring $(\cR,\tau)$, the set
\[  \cR^\tau := \{ x\in \cR \mid \tau(x)=x \}, \]
 is called the \emph{ring of invariants} of $\cR$. It is indeed a subring of $\cR$.

A \emph{difference module} over $(\cR,\tau)$ is an $\cR$-module $N$ together with a $\tau$-semilinear endomorphism $\tau_N:N\to N$, i.e.~a homomorphism of additive groups such that $\tau_N(xn)=\tau(x)\tau_N(n)$ for all $x\in \cR$ and $n\in N$. The set
\[  N^\tau:= \{ n\in N \mid \tau_N(n)=n \}\]
is called the \emph{module of invariants} of $N$. This is an $\cR^\tau$-module.

Given two difference modules $(N,\tau_N)$ and $(M,\tau_M)$, their tensor product $N\otimes_\cR M$ is a difference module with difference operator $\tau_{N\otimes M}=\tau_N\otimes \tau_M$.

Further, we denote by
\[ \Hom_\cR^\tau(N,M)=\{ f\in \Hom_\cR(N,M)\mid \tau_M\circ f=f\circ \tau_N \} \] 
the set of \emph{difference homomorphisms}, i.e.~the set of those $\cR$-module homomorphism that are compatible with the $\tau$-actions.

For a difference module $(N,\tau_N)$ over $(\cR,\tau)$, and a difference extension $\cS\supseteq \cR$, the tensor product $N_\cS:=\cS\otimes_\cR N$ is a difference module over $\cS$.


\begin{exmp}
Our base case will be the rings $\LA\subset \CA$ with the endomorphism being the Frobenius twist $\tau$ given by $\tau(a\otimes c)=a\otimes c^q$. This endomorphism will be extended to localizations $(\CA)_{(\pfrak)}$ and continuously to the completions $(\CA)_{\pfrak}$ and $\TT$. Here $\pfrak$ denotes an arbitrary  non-zero prime ideal of $A$, or the extension thereof.
As the ideal $\pfrak\CA\subseteq \CA$ is stable under $\tau$, we also get an induced endomorphism on
$\CA/\pfrak(\CA)\cong \Fpp\otimes \CI$.

The difference modules that we consider are derived from the $A$-motive of an Anderson $A$-module $E$, i.e.~from the module
\[ \mot:=\mot(E)=\Hom_{\vs}(E,\GG_{a,L}). \]
It carries an $A$-action induced by the $A$-action on $E$, and an action of $\End_{\vs}(\GG_{a,L})\cong L\{\tau\}$ where (by abuse of notation) $\tau$ denotes the $q$-power Frobenius endomorphism on $\GG_{a,L}$. As both actions commute with each other, $\mot$ can be seen as an $(\LA)$-module, and it is easy to verify that the $\tau$-action is a semilinear endomorphism on $\mot$ with respect to the $\tau$-action on $\LA$ given above.
\end{exmp}

Assume that the difference module $N$ is free of finite rank as $\cR$-module, and let $\{ b_1,\ldots, b_r\}$ be an $\cR$-basis of $N$. Then there is a matrix $\Phi\in \Mat_{r\times r}(\cR)$ such that
\[  \tau_N(b_i)=\sum_{j=1}^r \Phi_{ij} b_j\] for all $i=1,\ldots,r$, or in matrix notation
\[  \tau_N(\svect{b}{r}) = \Phi\cdot \svect{b}{r}. \]
Of course, the semilinear map $\tau_N$ is fully determined by this matrix.

\begin{lem}\label{lem:N-dualizable}
Let $(N,\tau_N)$ be a difference module over $(\cR,\tau)$ which is free of finite rank as $\cR$-module, and let $\{ b_1,\ldots, b_r\}$ be an $\cR$-basis of $N$. Then the following are equivalent
\begin{enumerate}
\item \label{item:ess-surjective} The image $\tau_N(N)$ contains a basis of $N$.
\item \label{item:Phi-invertible} The matrix $\Phi$ satisfying $\tau_N(b_i)=\sum_{j=1}^r \Phi_{ij} b_j$ for all $i=1,\ldots,r$, is invertible.
\item \label{item:difference-dual} There exists a unique $\tau$-difference structure on $\dual{N}:=\Hom_\cR(N,\cR)$ such that the
evaluation homomorphism $\ev:\dual{N}\otimes_\cR N\to \cR, f\otimes n\mapsto f(n)$ 
is a difference homomorphism.
\end{enumerate} 
\end{lem}

\begin{proof}
We start by proving \eqref{item:ess-surjective}$\Leftrightarrow$\eqref{item:Phi-invertible}:\\
All elements in $\tau_N(N)$ are $\cR$-linear combinations of elements of $\tau_N(b_1),\ldots, \tau_N(b_r)$. Hence,
if $\tau_N(N)$ contains a basis of elements of $N$, then the images $\tau_N(b_1),\ldots, \tau_N(b_r)$ form a basis of $N$. Therefore, the matrix $\Phi$ is a base change matrix, i.e.~invertible. On the other hand, if $\Phi$ is invertible, then $\tau_N(b_1),\ldots, \tau_N(b_r)$ form a basis of $N$ which is contained in $\tau_N(N)$.

Next we show \eqref{item:Phi-invertible}$\Leftrightarrow$\eqref{item:difference-dual}:\\
Let $\{ \dual{b_1},\ldots, \dual{b_r}\}$ be the basis of $\dual{N}$ dual to $\{ b_1,\ldots, b_r\}$, i.e.~
satisfying $\dual{b_i}(b_j)=\delta_{ij}$  for all $i,j\in \{1,\ldots, r\}$, where $\delta_{ij}$ is the Kronecker delta given by $\delta_{ij}=1$ for $i=j$, and $\delta_{ij}=0$ for $i\ne j$.

Define a difference structure on $\dual{N}$ by setting 
\[ \tau_{\dual{N}}\left( \dual{b_1},\ldots,\dual{b_r}\right) = \left( \dual{b_1},\ldots,\dual{b_r}\right)\cdot \Psi \]
for some $\Psi\in \Mat_{r\times r}(\cR)$.
An easy calculation shows that
\[  \ev\left( \tau_{\dual{N}\otimes N}(\dual{b_i}\otimes b_j)\right)=\sum_{k=1}^r \Psi_{ik}\Phi_{kj}=(\Psi\Phi)_{ij}.\]
Therefore, the evaluation homomorphism $\ev$ is a difference homomorphism if and only if
\[ (\Psi\Phi)_{ij} = \tau\left( \ev(\dual{b_i}\otimes b_j)\right)=\tau(\delta_{ij})=\delta_{ij} \]
for all $i,j=1,\ldots, r$. This is equivalent to $\Psi\cdot \Phi=\one_r$.

So, if the given difference structure is such that $\ev$ is a difference homomorphism, then $\Psi$ is the inverse of $\Phi$, and in particular $\Phi\in \GL_r(\cR)$.
On the other hand, if $\Phi\in \GL_r(\cR)$, we can set $\Psi=\Phi^{-1}$ to provide the unique difference structure on $\dual{N}$ such that $\ev$ is a difference homomorphism.
\end{proof}

\begin{rem}
It is not hard to see that the difference structure on $\dual{N}$ in the lemma is given in such a way that its $\tau$-invariant elements consist of the $\tau$-equivariant homomorphisms, in formulas
\[ (\dual{N})^\tau = \Hom_\cR^\tau(N,\cR). \]
\end{rem}

Assume now that $(\cR,\tau)$ is a difference ring such that $\cR$ is a domain, and
let $\cF$ be the field of fractions of $\cR$. Then $\cF$ is naturally a difference ring extension of $\cR$.

Let $(N,\tau_N)$ be a difference module over $(\cR,\tau)$ which is locally free of finite rank as $\cR$-module.
Then the difference module $N_\cF:=\cF\otimes_\cR N$ is a free module over $\cF$, and we have an embedding
\[ \dual{N}=\Hom_\cR(N,\cR)\subseteq \Hom_\cR(N,\cF)=\Hom_\cF(N_\cF,\cF)=\dual{(N_\cF)}. \]

\begin{prop}\label{prop:N-dualizable}
Let $(\cR,\tau)$ be a difference ring such that $\cR$ is a domain, and let $\cF$ be the field of fractions of $\cR$.
Let $(N,\tau_N)$ be a difference module over $(\cR,\tau)$ which is locally free of finite rank as $\cR$-module.
If $\tau_N(N)$ generates $N$ as an $\cR$-module, then the difference structure on $\dual{(N_\cF)}$ given by Lemma \ref{lem:N-dualizable} restricts to a difference structure on $\dual{N}$.
\end{prop}

\begin{proof}
The $\cR$-module $\dual{N}$ inside $\dual{(N_\cF)}$ is characterized by
\[ \dual{N}=\{ f\in \Hom_\cR(N,\cF) \mid f(n)\in \cR\, \forall n\in N\}. \]
So we have to show that for all $f\in \dual{N}$ and $n\in N$, $\tau_{\dual{(N_\cF)}}(f)(n)\in \cR$.

Since, $\tau_N(N)$ generates $N$ as an $\cR$-module, we can write any $n\in N$ as $n=\sum_i x_i\tau_N(n_i)$ with $x_i\in \cR , n_i\in N$. Then by the defining property of  $\tau_{\dual{(N_\cF)}}$,

\begin{eqnarray*}
\tau_{\dual{(N_\cF)}}(f)(n) &=& \tau_{\dual{(N_\cF)}}(f)\left( \sum_i x_i \cdot \tau_N(n_i)\right) 
= \sum_i x_i\cdot  \tau_{\dual{(N_\cF)}}(f)\left( \tau_N(n_i)\right) \\
&\overset{\text{Lem.~\ref{lem:N-dualizable}\eqref{item:difference-dual}}}{=}& \sum_i x_i\cdot \tau_\cR\left( f(n_i)\right) \in \cR.
\end{eqnarray*} 
\end{proof}

\begin{defn}\label{def:N-dualizable}
A difference module $(N,\tau_N)$ over $(\cR,\tau)$ that is free as $\cR$-module and satisfies the equivalent conditions of Lemma \ref{lem:N-dualizable}, or that is locally free over a domain $\cR$, and satisfies the hypothesis of Proposition \ref{prop:N-dualizable}, is called a \emph{dualizable} difference module.
\end{defn}

\begin{rem}
In the unpublished manuscript of Wibmer and the author, Lemma \ref{lem:N-dualizable} is generalized to any $\cR$-modules that are projective of finite rank by suitable modifications of conditions \eqref{item:ess-surjective} and \eqref{item:Phi-invertible}, including the case of Proposition \ref{prop:N-dualizable}.
It would then be more appropriate to call \emph{dualizable modules} all those projective difference modules satisfying these modified conditions resp.~condition \eqref{item:difference-dual}, as these are the ones admitting a categorial dual in the category of all difference modules. However, as our difference modules in question will meet one of the conditions above, we will not trouble the reader with the more general statement here.
\end{rem}

\begin{lem}\label{lem:generated-by-invariants}
Let $(N,\tau_N)$ be a dualizable difference module over $(\cR,\tau)$ that is free as $\cR$-module. Then the following are equivalent
\begin{enumerate}
\item \label{item:N-tau-generated} $N$ has a basis consisting of elements of $N^\tau$.
\item \label{item:N-tau-dual} The dual module $\dual{N}$ has a basis consisting of elements of $(\dual{N})^\tau$.
\item \label{item:N-tau-basis} $N^\tau$ is a free $\cR^\tau$-module of rank $\rk_\cR(N)$, and every $\cR^\tau$-basis $\{b_1,\ldots,b_r\}$ of $N^\tau$ is a $\cR$-basis of $N$.
\item \label{item:N-tau-isomorphism} The inclusion $N^\tau\subset N$ induces an isomorphism of $\cR$-modules 
\[ \alpha_N: \cR\otimes_{\cR^\tau} N^\tau \to N. \]
\end{enumerate}
\end{lem}

\begin{proof}
The equivalence of \eqref{item:N-tau-generated} and \eqref{item:N-tau-dual} is immediate from the description of the $\tau$-action on the dual module, and from recognizing that the bidual $\dual{(\dual{N})}$ is isomorphic to $N$ as difference modules.
The implications \eqref{item:N-tau-basis}$\Rightarrow$\eqref{item:N-tau-isomorphism} and \eqref{item:N-tau-isomorphism}$\Rightarrow$\eqref{item:N-tau-generated} are also trivial. So it remains to show the implication \eqref{item:N-tau-generated}$\Rightarrow$\eqref{item:N-tau-basis}.

So let $\{b_1,\ldots,b_r\}$ be an $\cR$-basis of $N$ consisting of $\tau$-invariant elements, and let $n\in N$. As $\{b_1,\ldots,b_r\}$ is a basis, there are unique $x_1,\ldots, x_r\in \cR$ such that $n=\sum_{j=1}^r x_jb_j$. Then
\begin{eqnarray*}
n \in N^\tau &\Leftrightarrow & \sum_{j=1}^r \tau(x_j)b_j = \tau_N(n)=n=\sum_{j=1}^r x_jb_j \\
&\Leftrightarrow & x_j\in \cR^\tau \quad \text{ for all }j=1,\ldots, r.
\end{eqnarray*}
Hence, $N^\tau = \bigoplus_{j=1}^r \cR^\tau b_j$ is a free $\cR^\tau$-module with basis $\{b_1,\ldots,b_r\}$.

If $\{b'_1,\ldots, b'_r\}$ is another $\cR^\tau$-basis of $N^\tau$, then there is a base change matrix $C\in \GL_r(\cR^\tau)$. In particular, $\{b'_1,\ldots, b'_r\}$ is also an $\cR$-basis of $N$.
\end{proof}

\section{Anderson $A$-modules over completions}

In this section, we consider elements which should be seen as $\TT$-points and as $(\CA)_\pfrak$-points of the Anderson module $E$.

\begin{defn}
After a choice of coordinate system $\kappa:E\cong \GG_{a,L}^d$, we have a norm on $E(\CI)$ induced by the maximum norm on $\CC_\infty^d$, and define
\[ E(\TT):= \left\{ \sum_{a\in B(A)} a\otimes e_a \bigmid e_a\in E(\CI),\, \lim_{\deg(a)\to \infty} \norm{e_a}=0 \right\} \]
as the completion of $A\otimes E(\CI)$ with respect to the Gauss norm. Although the norm depends on the coordinate system $\kappa$, the space itself doesn't (see also \cite[footnote on page 10]{qg-am:sfgtshrd}).
Further, for prime ideals $0\neq \pfrak\lhd A$, we set
\[ \Epfrak:= \varprojlim_{n} \left((A/\pfrak^{n+1})\otimes E(\CI)\right). \]
\end{defn}

\begin{rem}
\begin{enumerate}
\item On $A\otimes E(\CI)$, we have two different actions of $A$, namely one induced by the action via $\phi$ on the second tensor factor, and the other by multiplication on the first tensor factor. These induce two $A$-actions on the completions $E(\TT)$ as well as $\Epfrak$.

 By abuse of notation, we will
denote the action on the second tensor factor by $\phi$, too. Hence, 
\[\phi_a\left(\sum_i a_i\otimes e_{i}\right):=\sum_ia_i\otimes  \phi_a(e_i)\]
and
\[ a\cdot  \left(\sum_i a_i\otimes e_{i}\right) := \sum_i (aa_i)\otimes e_{i} \]
for $a\in A$ and $\sum_i a_i\otimes e_i\in A\otimes E(\CI)$.
On the space $\Epfrak$, the first action even extends to an action of $A_\pfrak$ turning $\Epfrak$ into an $A_\pfrak \otimes A$-module.

In the notation as power series $\sum_{n\geq 0} \up{n}f_n$ with $f_n\in \Fpp\otimes E(\CI)$, the action $\phi_a$ on $\Epfrak$ is still given via $\id_\Fpp \otimes \phi_a$ on the $f_n$. The $A_\pfrak$-action is given as
\[ \left(\sum_j c_j\up{j}\right) \cdot\left(\sum_i \up{i}f_i \right) = \sum_{n\geq 0}  \up{n} \left(\sum_{k=0}^nc_{n-k}\cdot  f_k \right) \]
for $\sum_i \up{i} f_i \in\Epfrak$ and $\sum_j c_j\up{j}\in \Fpp\ps{\up{}}$, where
$c_{n-k}\cdot f_k$ is just the product in the first tensor factor in $\Fpp \otimes E(\CI)$
\end{enumerate}

\end{rem}

\begin{defn}\label{def:H-hat-and-H}
For all $0\neq \pfrak\lhd A$ prime ideals, we define
\[ \hat{H}_\pfrak:=\hat{H}_{E,\pfrak}:= \{  h\in\Epfrak \mid \forall a\in A: \phi_a(h)=a\cdot h \}. \]
Further
\[  \fsf(E):=\{ h\in E(\TT) \mid \forall a\in A: \phi_a(h)=a\cdot h \}. \]
Hence, they are the subsets of $\Epfrak$ and of $E(\TT)$, respectively, on which the two $A$-actions coincide. The second space is called the space of special functions in \cite{qg-am:sfgtshrd} explaining the notation $\fsf(E)$.
\end{defn}

\section{Comparison Homomorphisms}\label{sec:homomorphisms}

The task of this section is to establish the following diagram.

\[  \xymatrix{ & & T_\pfrak(E)
 \ar@{<-}[d]^{\cong }_{\jp} &  \\  
& & \hat{H}_\pfrak \ar[r]^(.28){\ip}_(.28){\cong }
 & \Hom_{\LA}^\tau(\mot, (\CA)_\pfrak) \\
 \fd_{A/\Fq[u]}\cdot \Lambda_E \ar[rr]^(.6){\delta_u}_(.6){\cong } &&  \fsf(E) \ar[r]^(.4){\iota}_(.4){\cong } \ar@{^{(}->}[u]
 & \Hom_{\LA}^\tau(\mot, \TT)\ar@{^{(}->}[u]
}\] 
Here as before, $\pfrak$ is a prime ideal of $A$, and 
$\hat{H}_\pfrak$ and $\fsf(E)$ are the objects defined in Definition \ref{def:H-hat-and-H}.
Further, $\Lambda_E$ denotes the period lattice of $E$ and $\fd_{A/\Fq[u]}$ denotes the different ideal of the extension $A/\Fq[u]$ for some separating element $u\in A$. Finally,
$\mot:=\mot(E)=\Hom_{\vs}(E,\GG_{a,L})$ denotes the $A$-motive over $L$ associated to $E$, and the superscript $\tau$ at the Hom-sets stands for those homomorphisms which commute with the $\tau$-action (comp.~Section \ref{subsec:dualizable-difference-modules}).

The isomorphisms in the bottom row were already established in \cite{qg-am:sfgtshrd}, where the first isomorphism depended on the element $u$. In \cite{qg-am:rsfaacg}, it was shown that this isomorphism can be turned into a natural isomorphism $\Lambda_E\to \Omega_{A/\Fq}\otimes \fsf(E)$ whose inverse is given by taking residues.

Actually, the isomorphism $\jp$ will also depend on a uniformizer $\up{}\in A$ of the prime ideal $\pfrak$, but it stems from a natural isomorphism $\hat{\Omega}_{A_\pfrak/\Fpp}\otimes_{A_\pfrak}\hat{H}_\pfrak \xrightarrow{\tilde{\jp}} T_\pfrak(E)$ by composing with the isomorphism $A_\pfrak\to\hat{\Omega}_{A_\pfrak/\Fpp}, f\mapsto f d\up{ }$.

As we will see in Section \ref{sec:H-in-hat-H}, the injection $\fsf(E)\hookrightarrow \hat{H}_\pfrak$ will extend to an injective  $A_\pfrak$-homomorphism $A_\pfrak\otimes_A \fsf(E)\to \hat{H}_\pfrak$ which is an isomorphism, if $E$ is uniformizable.

\bigskip

We start by providing the inclusion $\fsf(E)\hookrightarrow \hat{H}_\pfrak$.

\begin{prop}\label{prop:H-in-H_pfrak} \ 
We have a natural inclusion of $A\otimes A$-modules $E(\TT)\hookrightarrow\Epfrak$.
In particular this induces an embedding $\fsf(E)\hookrightarrow \hat{H}_\pfrak$.
\end{prop}

\begin{proof}
The inclusion $E(\TT)\hookrightarrow\Epfrak$ is proven in the same manner as the inclusion $\TT\to (\CA)_\pfrak$ in Proposition~\ref{prop:inclusion-of-tt-in-padic}, since after a choice of coordinate system, $E(\CI)$ is isomorphic to $\CC_\infty^d$ as normed $\Fq$-vector space. Compatibility with the $A\otimes A$-action is clear, and hence this inclusion restricts to an inclusion of $A$-modules $\fsf(E)\hookrightarrow \hat{H}_\pfrak$.
\end{proof}

The inclusion $\Hom_{\LA}^\tau(\mot, \TT)\hookrightarrow \Hom_{\LA}^\tau(\mot, (\CA)_\pfrak)$ is directly obtained
via the natural inclusion $\TT\to (\CA)_\pfrak$ in Proposition~\ref{prop:inclusion-of-tt-in-padic}, so we continue by giving
the isomorphism $\ip$, and its restriction to the isomorphism $\iota$.

\begin{thm}\label{thm:h-isom-M-tate-dual}
There is a natural isomorphism of $A_\pfrak$-modules
\[  \ip:\hat{H}_\pfrak \longrightarrow \Hom_{\LA}^\tau( \mot, (\CA)_\pfrak) \]
which restricts to an isomorphism of $A$-modules
\[ \iota: \fsf(E)\longrightarrow \Hom_{\LA}^\tau( \mot, \TT). \]
\end{thm}

\begin{proof} (Similar to $t$-adic case in \cite[Theorem~3.9]{am:ptmsv})\\
The natural homomorphism of $\Fq$-vector spaces
\begin{equation}\label{eq:iso-E-to-M-dual}
     E(\CI) \longrightarrow \Hom_{L}^{\tau}(\mot, \CI), e\mapsto
\left\{ \mu_e: m\mapsto m(e) \right\}
\end{equation}
is an isomorphism, since after a choice of coordinate system $E(\CI)\cong \CC_\infty^d$ the latter is isomorphic to the bidual vector space 
$$E(\CI)^{\vee\vee}
=\Hom_{\CI}(\Hom_{\CI}(E(\CI),\CI),\CI)$$ of $E(\CI)$.
The homomorphism \eqref{eq:iso-E-to-M-dual} is even compatible with the $A$-action on $E$ via $\phi$ and the $A$-action on 
$\mu\in \Hom_{L}^{\tau}(\mot, \CI)$ via $(a\cdot \mu)(m)=
\mu(m\circ \phi_a)$ for all $m\in \mot$, $a\in A$.

The isomorphism of $A$-modules \eqref{eq:iso-E-to-M-dual} induces
isomorphisms of $A_\pfrak\otimes A$-modules
\[
(A/\pfrak^{n+1})\otimes  E(\CI)\longrightarrow  (A/\pfrak^{n+1})\otimes  \Hom_{L}^{\tau}(\mot, \CI)  = \Hom_{L}^{\tau}\left(\mot,  (A/\pfrak^{n+1})\otimes \CI\right)\]
by linear extension. Passing to the inverse limit and recognizing that $\Hom$ commutes with inverse limits, one obtains an isomorphism
\begin{eqnarray*}
\Epfrak &\longrightarrow & \varprojlim_n \Hom_{L}^{\tau}\left(\mot, (A/\pfrak^{n+1})\otimes  \CI\right)= \Hom_{L}^{\tau}(\mot, (\CA)_\pfrak).
\end{eqnarray*}
 By compatibility with the $A$-actions, the image of $\hat{H}_\pfrak\subseteq \Epfrak$ are exactly those
homomorphisms $\mu:\mot\to  (\CA)_\pfrak$ for which $\mu(m\circ \phi_a)=\mu(m)\cdot a$ for all $m\in \mot$ and all $a\in A$, i.e.~the $\LA$-linear ones, inducing the isomorphism
\[  \ip:\hat{H}_\pfrak \longrightarrow \Hom_{\LA}^\tau( \mot, (\CA)_\pfrak). \]
If we choose, an $\Fq$-basis $B(A)$ of $A$ as in Remark \ref{rem:TT-in-padic}, then we see that
$\sum_{a\in B(A)} a\otimes e_a\in\Epfrak$ is actually in $E(\TT)$ if and only if
 $\lim\limits_{a\to \infty} \norm{e_a}=0$. This is the case if and only if for all $m\in \mot$,
$\lim\limits_{a\to \infty} \betr{m(e_a)}=0$. Hence, the isomorphism
$\Epfrak\to \Hom_{L}^{\tau}(\mot, (\CA)_\pfrak)$ restricts to an isomorphism
$E(\TT) \to \Hom_{L}^\tau( \mot, \TT).$
As $\fsf(E)=E(\TT)\cap \hat{H}_\pfrak$ and $\Hom_{\LA}^\tau( \mot, \TT)=\Hom_{L}^\tau( \mot, \TT)\cap
\Hom_{\LA}^\tau( \mot, (\CA)_\pfrak)$, we get the desired isomorphism
\[ \iota: \fsf(E)\longrightarrow \Hom_{\LA}^\tau( \mot, \TT). \qedhere \]
\end{proof}

\begin{rem}
The isomorphism $\Epfrak\to  \Hom_{L}^{\tau}(\mot, (\CA)_\pfrak)$ is given in the $\up{}$-expansion as:
 \[
 \sum_i \up{i}e_i  \mapsto  \qquad \left\{
 \sum_i \up{i}\mu_{e_i}   : m\mapsto \sum_i \up{i}(\id_{\Fpp}\otimes m)(e_i) \right\}.
\]
\end{rem}

We end this section by providing the isomorphism $\jp: \hat{H}_\pfrak \to T_\pfrak(E)$.

\begin{thm}\label{thm:isom-jp} \
\begin{enumerate}
\item \label{item:general-isoms} We have isomorphisms of $A_\pfrak\otimes A$-modules
\[  \hat{\Omega}_{A_\pfrak/\Fpp}\otimes_{A_\pfrak} \Epfrak \to \Hom_{\Fpp}\left(K_\pfrak/A_\pfrak, \Fpp\otimes E(\CI)\right)\to \Hom_{\Fq}\left(K_\pfrak/A_\pfrak, E(\CI)\right), \]
where the first map is given by 
\[ \omega \otimes \left( \sum_{a\in B(A)} a\otimes e_a\right)\quad\longmapsto\quad \left\{ x\mapsto \sum_{a\in B(A)} \res_\pfrak(xa\omega)\otimes e_a \right\},\]
for an $\Fq$-basis $B(A)$ of $A$ as in Remark \ref{rem:TT-in-padic},  
and the second map is given by
\[ f\mapsto (\tr \otimes \id)\circ f\] where $\tr:\Fpp\to \Fq$ denotes the trace map.
\item \label{item:tilde-jp} These isomorphisms restrict to isomorphisms
\[ \tilde{\jp}: \hat{\Omega}_{A_\pfrak/\Fpp}\otimes_{A_\pfrak} \hat{H}_\pfrak \to \Hom_{\Fpp\otimes A}\left(K_\pfrak/A_\pfrak, \Fpp\otimes E(\CI)\right)\to \Hom_{A}\left(K_\pfrak/A_\pfrak, E(\CI)\right) = T_\pfrak(E). \]
\item \label{item:jp} After the choice of a uniformizer $\up{}$ at $\pfrak$, we obtain an isomorphism of $A_\pfrak$-modules
\[ \jp: \hat{H}_\pfrak \to T_\pfrak(E), h\mapsto \tilde{\jp}(d\up{}\otimes h).\]
\end{enumerate}
\end{thm}

\begin{proof}
\eqref{item:general-isoms} It is not hard to check that both maps are well-defined and $A_\pfrak\otimes A$-linear.
An inverse to the map $f\mapsto (\tr\otimes \id)\circ f$ is given by
\[ g\longmapsto \quad \left\{ x \mapsto \sum_{i=1}^\degp \alpha_i\otimes g(\check{\alpha}_i x) \right\},\]
where $\alpha_1,\ldots, \alpha_\degp$ is an $\Fq$-basis of $\Fpp$, and $\check{\alpha}_1,\ldots, \check{\alpha}_\degp$ is the dual basis of $\Fpp$ with respect to $\tr$, i.e.~satisfying $\tr(\check{\alpha}_i\alpha_j)=1$ if $i=j$, and $\tr(\check{\alpha}_i\alpha_j)=0$ if $i\ne j$. (Take into account that this choice implies $\sum_{i=1}^\degp \tr(\alpha_i)\check{\alpha}_i=1=\sum_{i=1}^\degp \tr(\check{\alpha}_i)\alpha_i$.)

An inverse to the first map is given by
\[ f \mapsto d\up{}\otimes \left( \sum_{j=0}^\infty \up{j}\cdot f(\up{-j-1})\right), \]
where as before $\up{}\in A$ is a uniformizer at $\pfrak$.

Part \eqref{item:tilde-jp} is obtained by taking in each module the subset of elements on which both $A$-actions coincide, and part \eqref{item:jp} is clear from that.
\end{proof}

\section{Relation of $\fsf(E)$ in $\hat{H}_\pfrak$}\label{sec:H-in-hat-H}

In this section, we provide results on the structure of $\fsf(E)$, and of the image of the inclusion $\fsf(E)\to \hat{H}_\pfrak$.
This is a generalization of \cite[Section 3.1]{am:ptmsv}

We also deduce that uniformizable Anderson modules (even if they are not abelian) are regular in the sense of Yu.

\begin{defn}
For a prime ideal $\pfrak$, we define the $\pfrak$-rank of $E$, $\trk{\pfrak}(E)$ to be the dimension of the $\pfrak$-torsion 
$E[\pfrak]$ as a vector space over $\FF_\pfrak:=A/\pfrak$.
\end{defn}

\begin{rem}\label{rem:p-rk-rk-Hp}
Since, $\hat{H}_\pfrak/\pfrak\hat{H}_\pfrak \isom E[\pfrak]$, we have
\[ \trk{\pfrak}(E) = \rk_{A_\pfrak} \hat{H}_\pfrak. \]
\end{rem}

\begin{lem}\label{lem:H-locally-free}
For any Anderson module $E$ (even non-abelian, non-uniformizable), the space of special functions $\fsf(E)$ and the period lattice $\Lambda_E$ are finitely generated torsion-free $A$-modules, hence even locally free.
\end{lem}

\begin{proof}
Since the left $A$-action on $E(\TT)$ is obviously torsion-free, the $A$-action on $\fsf(E)$ (which is the restricted left $A$-action) is also torsion-free, and hence $\fsf(E)$ is a torsion-free $A$-module.

As $\fsf(E)$ is isomorphic to the period lattice $\Lambda_E$ up to multiplication with the ideal $\fd_{A/\Fq[u]}$, it suffices to show finite generation for $\Lambda_E$.

For $t\in A\setminus \Fq$, we can consider $E$ as a $t$-module via $\phi|_{\Fq[t]}$ having the same exponential function, and hence the same period lattice.
By \cite[Proposition 2.2(2)]{am:ptmsv}, $\Lambda_E$ is a free $\Fq[t]$-module of finite rank.
Hence, it is also finitely generated over $A$.
\end{proof}

\begin{lem}\label{lem:H-saturated}
One has
\[ \fsf(E)\cap \pfrak\hat{H}_\pfrak = \pfrak\,  \fsf(E). \]
\end{lem}

\begin{proof}
We first show $\fsf(E)\cap \pfrak^n\hat{H}_\pfrak = \pfrak^n \fsf(E)$ where $n\in \NN$ is such that $\pfrak^n\lhd A$ is principal. Actually, in both equations the inclusions $\supseteq$ are clear and we only show the reverse inclusions.

Let $f\in \pfrak^n$ be a generator, and choose lifts $\{b_1,\ldots, b_m\}\subset A$ of an $\Fq$-basis of $A/(f)$. Then
\[B(A):=\{ b_if^k\mid k\geq 0, i=1,\ldots, m\}\] is an $\Fq$-basis of $A$.

Let $h=\sum_{a\in B(A)} a \otimes h_a\in \fsf(E)\cap \pfrak^n\hat{H}_\pfrak =\fsf(E)\cap f\hat{H}_\pfrak$.
We claim that
\[  \tilde{h}:= \sum_{a\in B(A)} a\otimes h_{fa}\in E(\TT) \]
is indeed in $\fsf(E)$ and satisfies $f\tilde{h}=h$. This shows that $h\in f \fsf(E)=\pfrak^n \fsf(E)$.

One has $h_{b_i}=0$ for all $i\in 1,\ldots, m$, since $h\in f\hat{H}_\pfrak$, and therefore
\[ h= \sum_{a\in B(A)\setminus \{b_1,\ldots, b_m\}} a \otimes h_a
=\sum_{a\in B(A)} (fa)\otimes h_{fa} =f\tilde{h}. \]

Furthermore, for all $x\in A$:
\[ f\cdot \phi_x(\tilde{h})=f\cdot \sum_{a\in B(A)} a\otimes \phi_x(h_{fa})
= \sum_{a\in B(A)} fa\otimes \phi_x(h_{fa})=\phi_x(h)=xh=fx\tilde{h},\]
hence $f(\phi_x-x)\tilde{h}=0$. As the left-$A$-action on $E(\TT)$ is torsion-free, this implies 
$(\phi_x-x)\tilde{h}=0$, i.e.~$\tilde{h}\in \fsf(E)$.

\medskip

Now let $h\in \fsf(E)\cap \pfrak \hat{H}_\pfrak$, and let $\up{}\in \pfrak\setminus \pfrak^2$.
Then by the first part, $\up{n-1}h\in \fsf(E)\cap \pfrak^n \hat{H}_\pfrak=f \fsf(E)$. Therefore, there exists $\tilde{h}\in \fsf(E)$ such that $\up{n-1}h=f\tilde{h}$.\\
As $\up{n}\in \pfrak^n\setminus \pfrak^{n+1}$ there is some element $c\in A\setminus \pfrak$ such that $\up{n}=cf$, and we obtain
\[ cfh=\up{n}h=\up{}f\tilde{h}, \text{ i.e.}\quad ch=\up{}\tilde{h}\in \fsf(E).\]
Finally, as $c\notin \pfrak$, there are $a\in A$ and $x\in \pfrak$ such that $ac+x=1$, and we conclude
\[ h=(ac+x)h=a\up{}\tilde{h}+xh\in \pfrak\, \fsf(E). \qedhere \]
\end{proof}

\begin{prop}\label{prop:H-direct-summand}
The extension of the inclusion $\fsf(E)\hookrightarrow \hat{H}_\pfrak$ to an $A_\pfrak$-linear homomorphism $A_\pfrak\otimes_A \fsf(E)\to \hat{H}_\pfrak$ is still injective, and
the image of this map is a direct summand of $\hat{H}_\pfrak$.
In particular, the map $A_\pfrak\otimes_A \fsf(E)\to \hat{H}_\pfrak$ is an isomorphism if and only if $\rk_A(\fsf(E))=\rk_{A_\pfrak}\hat{H}_\pfrak$.
\end{prop}

\begin{proof}
As $\fsf(E)$ is torsion free, the extension $A_{(\pfrak)}\otimes_A \fsf(E)\to \hat{H}_\pfrak$ is still injective where $A_{(\pfrak)}$ denotes the (non-complete) local ring of $A$ at $\pfrak$.

By Lemma \ref{lem:H-locally-free}, $A_{(\pfrak)}\otimes_A \fsf(E)$ is free of finite rank, and we can choose an $A_{(\pfrak)}$-basis $\{h_1,\ldots, h_s\}$. Of course, we can even choose $h_i\in \fsf(E)$.

We will show that their images $\bar{h}_1,\ldots, \bar{h}_s\in \hat{H}_\pfrak/\pfrak\hat{H}_\pfrak$ are $\Fpp$-linearly independent. From this the claim follows, since by Nakayama's lemma a set $\{f_1,\ldots,f_r\}$ of elements in $\hat{H}_\pfrak$ form an $A_\pfrak$-basis if and only if their residues modulo $\pfrak$ form an $\Fpp$-basis of $\hat{H}_\pfrak/\pfrak\hat{H}_\pfrak=E[\pfrak]$.

For the contrary, assume that there is a non-trivial relation $\sum_{j=1}^s a_jh_j\in \pfrak\hat{H}_\pfrak$ with $a_1,\ldots, a_s\in A_\pfrak$, not all lying in $\pfrak A_\pfrak$. Take $c_1,\ldots, c_s\in A$ s.t.~$c_i\equiv a_i\mod \pfrak$ for all $i$. Then 
\[ \sum_{j=1}^s c_jh_j\in \pfrak\hat{H}_\pfrak \cap \fsf(E)=\pfrak \fsf(E)\subseteq \up{}(A_{(\pfrak)}\otimes_A \fsf(E)),\]
by Lemma \ref{lem:H-saturated}.

As $h_1,\ldots, h_s$ is a basis of $A_{(\pfrak)}\otimes_A \fsf(E)$, there are $d_1,\ldots, d_s\in A_{(\pfrak)}$ such that $\sum_{j=1}^s c_jh_j=\up{}\cdot (\sum_{j=1}^s d_jh_j)$, i.e.~
\[  0=\sum_{j=1}^s c_jh_j - \up{}\cdot (\sum_{j=1}^s d_jh_j)=\sum_{j=1}^s (c_j-\up{}d_j)h_j, \]
contradicting the assumption that $h_1,\ldots, h_s$ are linearly independent. 
\end{proof}

\begin{cor}\label{cor:uniformizable-implies-constant-prank}
If $E$ is uniformizable, then every $A_{(\pfrak)}$-basis of $A_{(\pfrak)}\otimes_A \fsf(E)$ is an $A_\pfrak$-basis of $\hat{H}_\pfrak$.
In particular, $\hat{H}_\pfrak$ has an $A_\pfrak$-basis consisting of 
special functions, and
we have
$\rk_A(\fsf(E))=\rk_{A_\pfrak}(\hat{H}_\pfrak)= \trk{\pfrak}(E)$. The latter implies that
the $\pfrak$-rank is independent of the prime $\pfrak$.
\end{cor}

\begin{proof}
We are going to show that $\rk_A(\Lambda_E)=\dim_{\Fpp} E[\pfrak]$. The rest then follows, since
$\rk_A(\fsf(E))=\rk_A(\Lambda_E)$ by the isomorphism $\delta_{u}$, and $\trk{\pfrak}(E)= \dim_{\Fpp} E[\pfrak]= \rk_{A_\pfrak}\hat{H}_\pfrak$.

By definition of $\Lambda_E$,
the exponential map $\e$ induces an injection
\[ \e:\LieE(\CI)/\Lambda_E \hookrightarrow E(\CI). \]
Using the functional equation of the exponential map, we see that the image $\e(\alpha)$ of an element $\alpha\in \LieE(\CI)$ lies in $E[\pfrak]$, if and only if $\alpha\in \pfrak^{-1}\Lambda_E$.
Hence, if $E$ is uniformizable, i.e.~$\e$ is surjective, then we obtain an isomorphism of $\Fpp$-vector spaces $\pfrak^{-1}\Lambda_E/\Lambda_E\to E[\pfrak]$, in particular 
\[\dim_{\Fpp} \left(\pfrak^{-1}\Lambda_E/\Lambda_E \right)=\dim_{\Fpp} E[\pfrak]=\trk{\pfrak}(E).\]

As $\Lambda_E$ is a locally free $A$-module, we further have
$\dim_{\Fpp} (\pfrak^{-1}\Lambda_E/\Lambda_E)=\rk_A(\Lambda_E)$.
\end{proof}

The previous corollary shows an unexpected connection between uniformizability and regularity as defined by Yu in \cite{jy:ahdm}.

\begin{defn} (cf.~\cite[p.~218]{jy:ahdm})
An Anderson $A$-module $E$ is said to be \emph{regular}, if there exists an integer $r\geq 0$ such that for all ideals $\afrak\lhd A$, the torsion $E[\afrak]$ is an $A/\afrak$-module of rank $r$. 
\end{defn} 

\begin{cor}\label{cor:uniformizable-implies-regular}
Every uniformizable Anderson $A$-module $E$ is regular.
\end{cor}

\begin{proof}
Let $E$ be uniformizable, and set $r=\rk_A(\fsf(E))$. Then by Corollary~\ref{cor:uniformizable-implies-constant-prank}, we have $\rk_{A_\pfrak}(\hat{H}_\pfrak)=r$ for all primes $\pfrak$. Hence also 
$\rk_{A/\pfrak^n}(E[\pfrak^n])=r$ for all primes $\pfrak$ and all $n\geq 1$.

Further, for ideals $\afrak,\bfrak\lhd A$ which are relatively prime, one has $E[\afrak\cdot \bfrak]=E[\afrak]\times E[\bfrak]$ as modules over $A/(\afrak\bfrak)\cong (A/\afrak) \times (A/\bfrak)$.
Hence, the claim follows by induction on the number of distinct prime factors using the prime decomposition of ideals.
\end{proof}

\section{Torsion points as values of hyperderivatives}\label{sec:torsion-points-as-values}

For a special function $h\in \fsf(E)$, its evaluation at $\pfrak$ is a Gauss-Thakur sum
in $\Fpp\otimes E[\pfrak]$ (see \cite{qg-am:sfgtshrd}).
In a similar manner as in \cite{ba-fp:ugtsls} and \cite{am-rp:iddbcppte} for the prime-power torsion of the Carlitz module, we are going to obtain also the $\pfrak^{n+1}$-torsion as evaluation at $\pfrak$.

In Example \ref{exmp:hyperderivations-wrt-pfrak}, we already introduced the hyperdifferential operators $(\hde{\up{}}{n})_{n\geq 0}$ with respect to $\up{}$ on $A_{(\pfrak)}$.
By $\CI$-linear extension, they also define hyperderivations on $A_{(\pfrak)}\otimes \CI$, and by Lemma \ref{lem:hd-p-adic-continuous} can be extended to $(A\otimes \CI)_\pfrak$.
After choice of coordinate system, we get hyperderivations on $(A\otimes E(\CI))_\pfrak=\Epfrak$, and the hyperderivations do not depend on that choice. 

\begin{prop}\label{prop:formula-for-hyperderivation}
Let $(\partial_{\up{}}^{(n)})_{n\geq 0}$  be the hyperderivation on $\Epfrak$ given above. Then for 
$\sum_{k=0}^\infty \up{k}e_k\in \Epfrak $ with $e_k\in \Fpp \otimes E(\CI)$. One has
\[ \partial_{\up{}}^{(n)}\left( \sum_{k=0}^\infty \up{k}e_k\right)=
\sum_{k=n}^\infty \binom{k}{n} \up{k-n}e_k. \]
\end{prop}

\begin{proof}
As the hyperdifferential operators are trivial on $\Fq$, and $\Fpp$ is finite over $\Fq$, the hyperdifferential operators are also trivial on $\Fpp$. Furthermore by definition, we have
$\partial_{\up{}}^{(n)}(\up{k})= \binom{k}{n} \up{k-n}$.
The formula then follows by linearity and $\pfrak$-adic continuity.
\end{proof}

We recall that for $h\in \Epfrak$, we defined the evaluation of $h$ at $\pfrak$ -- denoted by $h(\pfrak)$ -- to be the image of $h$ under the residue map $\Epfrak/\pfrak\Epfrak\isom \Fpp\otimes E(\CI)$.

\begin{prop}\label{prop:special-values}
Let $(E,\phi)$ be a uniformizable Anderson $A$-module, and $h_1,\ldots, h_r\in \fsf(E)$ an $A_{(\pfrak)}$-basis of $A_{(\pfrak)}\otimes_A \fsf(E)$ which we view as elements of $\hat{H}_\pfrak$ via the inclusion $\fsf(E)\hookrightarrow \hat{H}_\pfrak$.
For each $n\in \NN$,
\begin{enumerate}
\item the evaluations of $\partial_{\up{}}^{(n)}(h_i)$ at $\pfrak$ for $(i=1,\ldots, r)$ are in $\Fpp\otimes E[\pfrak^{n+1}]$, and
\item their images under $(\tr\otimes \id)$, i.e.~the elements $(\tr\otimes \id)\left(\partial_{\up{}}^{(n)}(h_i)(\pfrak)\right)$ for $(i=1,\ldots, r)$, form an $(A/\pfrak^{n+1})$-basis of $E[\pfrak^{n+1}]$.
\end{enumerate}
\end{prop}

\begin{proof}
We use the explicit inverse of the isomorphism $\jp$ given in the proof of Theorem~\ref{thm:isom-jp}.
From this we deduce that setting $f_j:=\jp(h_j)\in T_\pfrak(E)$, we can write $h_j$ as
\[ h_j = \sum_{j=0}^\infty \up{j} \sum_{i=1}^{\degp} \alpha_i\otimes f_j(\check{\alpha}_i \up{-j-1}). \]
Using Proposition~\ref{prop:formula-for-hyperderivation}, we see that 
\[ \partial_{\up{}}^{(n)}(h_i)(\pfrak) = \sum_{i=1}^{\degp} \alpha_i\otimes f_j(\check{\alpha}_i \up{-n-1})\in \Fpp\otimes E[\pfrak^{n+1}].\]
Further,
\begin{eqnarray*}
(\tr\otimes \id)\left(\partial_{\up{}}^{(n)}(h_i)(\pfrak)\right) &=& \sum_{i=1}^{\degp} \tr(\alpha_i)\cdot f_j(\check{\alpha}_i \up{-n-1})\\
&=& f_j\left( \sum_{i=1}^{\degp} \tr(\alpha_i) \check{\alpha}_i \up{-n-1}\right) = f_j(\up{-n-1}),
\end{eqnarray*}
using the properties of the bases $\{\alpha_i\}$ and $\{\check{\alpha}_i\}$ mentioned above.

By Corollary \ref{cor:uniformizable-implies-constant-prank}, $\{h_1,\ldots, h_r\}$ is an $A_\pfrak$-basis of $\hat{H}_\pfrak$.
Hence, $\{f_1,\ldots, f_r\}$ is an $A_\pfrak$-basis of $T_\pfrak(E)$, and therefore the values
$\{ f_1(\up{-n-1}),\ldots, f_r(\up{-n-1})\}$ are an $(A/\pfrak^{n+1})$-basis of $E[\pfrak^{n+1}]$.
\end{proof}

\begin{rem}
In the same manner as the $\up{}$-expansion depends on the choice of the uniformizer $\up{}$, also the hyperdifferential operators $ \partial_{\up{}}^{(n)}$ depend on the choice of $\up{}$.

In the next theorem, we will see that the dependence is not too bad. This means that up to scalars we always get the same basis of $E[\pfrak^{n+1}]$ modulo
$E[\pfrak^{n}]$. The next theorem is even stronger, and shows that we can choose ``almost any'' element $t\in A\setminus \Fq$ instead of a uniformizer $\up{}$.
\end{rem}

\begin{thm}\label{thm:torsion-as-special-values}
Let $(E,\phi)$ be a uniformizable Anderson $A$-module, and $h_1,\ldots, h_r\in \fsf(E)$ an $A_{(\pfrak)}$-basis of $A_{(\pfrak)}\otimes_A \fsf(E)$ which we view as elements of $\hat{H}_\pfrak$ via the inclusion $\fsf(E)\hookrightarrow \hat{H}_\pfrak$.
Further, let $t\in A\setminus \Fq$ such that $\hd{\up{}}{1}{t}\in A_{(\pfrak)}^\times$. 
Then
\begin{enumerate}
\item The hyperderivation $(\hdte{n})_{n\geq 0}$ on $\Fq[t]$ uniquely extends to $A_{(\pfrak)}$.
\item For all $h\in \hat{H}_\pfrak$ and $n\geq 0$,  one has $\partial_{t}^{(n)}(h)(\pfrak)\in \Fpp\otimes E[\pfrak^{n+1}]$ as well as
\[    \hdt{n}{h}(\pfrak) \equiv \left( \left.\hd{\up{}}{1}{t}\right.^{-n}\cdot
 \hd{\up{}}{n}{h}\right) (\pfrak) \mod \Fpp\otimes E[\pfrak^{n}]. \]
\item For all $n\geq 0$, the elements $(\tr\otimes\id)\left(\hd{t}{n}{h_i}(\pfrak)\right)$ for $i=1,\ldots, r$ form an $(A/\pfrak^{n+1})$-basis of $E[\pfrak^{n+1}]$. 

\end{enumerate}
\end{thm}

\begin{proof}
The first part is just the one-dimensional case of \cite[Theorem~4.5]{am:crfhdipm}, and the third part directly follows from the second part and Proposition~\ref{prop:special-values}.

The second part is clear for $n=0$, so we assume $n\geq 1$.
Using the chain rule for iterative derivations (see \cite[Proposition~7.2 \& Proposition~7.3]{ar:icac} or 
\cite[Theorem~4.1]{am:crfhdipm}), we obtain for any $h\in \Epfrak$:
\begin{eqnarray*}
  \hd{t}{n}{h} &=& \sum_{k=1}^n \hd{\up{}}{k}{h}\cdot \sum_{\substack{\substack{j_1,\ldots, j_k\geq 1\\ j_1+\ldots +j_k=n}}}  \hd{t}{j_1}{\up{}}\cdots  \hd{t}{j_k}{\up{}} \\
&=& \hd{\up{}}{n}{h}\cdot \left.\hd{t}{1}{\up{}}\right.^n
+ \sum_{k=1}^{n-1} \hd{\up{}}{k}{h}\cdot \sum_{\substack{j_1,\ldots, j_k\geq 1\\ j_1+\ldots +j_k=n}}  \hd{t}{j_1}{\up{}}\cdots  \hd{t}{j_k}{\up{}}
\end{eqnarray*}
If we evaluate at $\pfrak$ for $h\in \fsf(E)$, and use Proposition~\ref{prop:special-values} as well as 
$\hd{t}{1}{\up{}}=\hd{\up{}}{1}{t}^{-1}$, we obtain the desired congruence.
\end{proof}

\begin{rem}
In the polynomial case $A=\Fq[t]$ in \cite{am-rp:tcagfdte}, Perkins and me, we always considered hyperderivation with respect to $t$. The previous theorem justifies these considerations:
Namely, for any prime ideal $\pfrak\subseteq \Fq[t]$ generated by a polynomial $\up{}(t)$, its derivative $\hd{t}{1}{\up{}}$ is prime to $\up{}$. Hence, the valuation at $\pfrak$ of 
$\hd{\up{}}{1}{t}=\hd{t}{1}{\up{}}^{-1}\in K$ is zero. In other words,
$\hd{\up{}}{1}{t}\in A_{(\pfrak)}^\times$.
Therefore, by the previous theorem, the hyperderivation with respect to the indeterminate $t$ can be used for every prime $\pfrak$ in $\Fq[t]$. 

Furthermore, we can make the relation of the coefficients $c_{j,(n),l}\in E[\pfrak^{n+1}]$ given in \cite[Eq.~(3.1)]{am-rp:tcagfdte} to our evaluations
$\hdt{n}{h}(\pfrak)\in \Fpp\otimes  E[\pfrak^{n+1}]$ quite explicit.
Namely, after choosing a root $\zeta\in \Fpp$ of $\pfrak\in \Fq[t]$, we can write $\hdt{n}{h}(\pfrak)$ uniquely as
\[ \hdt{n}{h}(\pfrak) = \sum_{l=0}^{\degp-1} \zeta^l \otimes e_l  \]
for some $e_l\in E[\pfrak^{n+1}]$. Then for $h=\omega_j$, one has $e_l= c_{j,(n),l}$.
\end{rem}

\section{$A$-motives of non-abelian Anderson $A$-modules}

In this section, we will provide more insight into the structure of $A$-motives of non-abelian Anderson $A$-modules.

Recall that for an Anderson $A$-module $E$, we consider its $A$-motive
\[ \mot:=\mot(E)=\Hom_{\vs}(E,\GG_{a,L}) \]
where the $A$-action is induced by the $A$-action on $E$, and the $L\{\tau\}=\End_{\vs}(\GG_{a,L})$-action is induced by the action on $\GG_{a,L}$.

As $E$ is a vector space scheme isomorphic to $\GG_{a,L}^d$, the module $\mot(E)$ is always a free and finitely generated module over 
$L\{\tau\}=\End_{\vs}(\GG_{a,L})$ of dimension $d$. However, if $E$ is not abelian, $\mot(E)$ is by definition not finitely generated as $\LA$-module.

\subsection{Finiteness of the rank of the $A$-motive}

The next theorem shows that also in the non-abelian case, the structure of the $A$-motive $\mot(E)$ is not too bad, i.e.~after a certain localization of the ring $\LA$, we obtain a finitely generated free module.

\begin{thm}\label{thm:rational-A-motive-f-dim}
Let $\Quot(\LA)$ be the field of fractions of $\LA$.
The $A$-motive $\mot:=\mot(E)$ associated to $E$ has the following properties:
\begin{enumerate}
\item\label{item:1} $\Quot(\LA)\otimes_{\LA}\mot $ is a finite dimensional $\Quot(\LA)$-vector space. 
\item\label{item:2} There is an element $f\in \LA$ such that one has:
Let $F$ be the multiplicative subset generated by $f$ and all its images under (iterated) $\tau$-twists, then $(\locLA)\otimes_{\LA} \mot$ is a difference module which is finitely generated as $(\locLA)$-module.
\item\label{item:3} The element $f\in \LA$ and the corresponding set $F$ as in \eqref{item:2}, can be chosen such that $(\locLA)\otimes_{\LA} \mot$ is a free module.
\end{enumerate}
\end{thm}

\begin{defn}
We call the dimension of $\Quot(\LA)\otimes_{\LA}\mot$ as $\Quot(\LA)$-vector space (which is finite by the previous theorem) the \emph{rank of the $A$-motive $\mot$}, and denote it by $\rk(\mot)$.
\end{defn}

\begin{rem}
The previous theorem ensures that $\rk(\mot)$ is a natural number also in the non-abelian case. We will see in the Section \ref{subsec:equality-of-ranks} that for almost all primes $\pfrak$ of $A$, we even have $\rk(\mot)=\trk{\pfrak}(E)$. This is in accordance to the abelian case where this equality holds for all primes $\pfrak$.
\end{rem}

The proof of the first two parts is an adaption of ideas in \cite[Sect.~6]{am:aefam}.

\begin{proof}[Proof of Theorem \ref{thm:rational-A-motive-f-dim}]
Let $t\in A$, then $L[t]=\Fq[t]\otimes L\subseteq \LA$ is a finite ring extension, and $L(t)\subseteq \Quot(\LA)$ is a finite field extension. 
Hence, $\mot_{L(t)}:=L(t)\otimes_{L[t]} \mot=\Quot(\LA)\otimes_{\LA} \mot$, and for \eqref{item:1} and \eqref{item:2}, it suffices to show finite generation of $\Quot(\LA)\otimes_{\LA} \mot$ as a module over $L(t)$ and over $\locLt$ for a multiplicative subset $F$ generated by some $f\in L[t]$ and its $\tau$-twists.
We start with the proof of part \eqref{item:1}.
\begin{enumerate}
\item With respect to some $L\{\tau\}$-basis $\kappa_1,\ldots,\kappa_d$ of $\mot$, we can write
\[ t\cdot \svect{\kappa}{d} = D\cdot \svect{\kappa}{d} \]
for some $D\in \Mat_{d\times d}(L\{\tau\})$. Hence, the matrix $D-t\cdot \one_d\in L[t]\{\tau\}$ annihilates the $L\{\tau\}$-basis $\kappa_1,\ldots,\kappa_d$.

Let $s\in\NN$ be the maximum $\tau$-degree of entries of $D$.
If the matrix $D_s$ given by the coefficients of $\tau^s$ is invertible over $L(t)$, we are done, as this implies that we can write
\[ \tau^s \svect{\kappa}{d}=-D_s^{-1}\cdot \left( D-t\cdot \one_d-D_s\tau^s\right)\cdot \svect{\kappa}{d}, \]
and the $\tau$-degree of $D-t\cdot \one_d-D_s\tau^s$ is at most $s-1$.
Therefore all $\tau^s\kappa_j$ are $L(t)$-linear combinations of the $\tau^i\kappa_j$ with $i<s$ and $j=1,\ldots d$, and by twisting the equation by powers of $\tau$, we obtain, that all $\tau^k\kappa_j$
with $k\geq s$ are in the $L(t)$-span of the $\tau^i\kappa_j$ with $i<s$ and $j=1,\ldots d$. In particular, $\mot_{L(t)}$ is finitely generated as $L(t)$-module.\footnote{Actually, unless the $t$-module $E$ wasn't a product of $\GG_a$'s with trivial $t$-action, in the considered case, we would even have $D_s\in \Mat(L)$ for which it is known that the motive $\mot$ is abelian.}

If $D_s$ is not invertible, we are going to find a matrix $D'\in \Mat_{d\times d}(L(t)\{\tau\})$ such that 
$D'\cdot (D-t\one_d)$ has the property that its top $\tau$-coefficient matrix is invertible over $L(t)$.
Then we can conclude the finite generation as in the special case.

Let $L^{\per}=\bigcup_{n\geq 0} L^{1/p^n}\subseteq \bar{L}$ be the perfect closure of $L$.

Consider the matrix $\tau^{-s}\cdot (D-t\cdot \one_d)\in \Mat_{d\times d}(L^{\per}(t)\{\tau^{-1}\})$. As $L^{\per}(t)\{\tau^{-1}\}$ is a skew-polynomial ring whose endomorphism is even an automorphism, there are left and right division algorithms in this ring. Therefore, as in \cite{ga:tm}, Proposition~1.4.2, there are matrices $B,C\in \GL_{d}(L^{\per}(t)\{\tau^{-1}\})$ such that $B\tau^{-s}\cdot (D-t\cdot \one_d)C$ is a diagonal matrix. Multiplying from the left with a diagonal matrix $T$ whose diagonal entries are appropriate powers of $\tau$, we get
\[ TB\tau^{-s}\cdot (D-t\cdot \one_d)C \equiv S \mod \tau^{-1}\] with $S\in \GL_d(L^{\per}(t))$ (of course a diagonal matrix). Hence,
\[ CTB\tau^{-s}\cdot (D-t\cdot \one_d) \equiv CSC^{-1}\equiv S' \mod \tau^{-1}\] with $S'\in \GL_d(L^{\per}(t))$.

Finally, there is $r\in \NN$ such that $D':=\tau^{r}CTB\tau^{-s}\in \Mat_{d\times d}(L(t)\{\tau\})$, i.e.~that no negative $\tau$-powers remain, and that all $\tau$-coefficients are in $L(t)$.
By construction, this $D'$ satisfies the desired property, as the leading coefficient matrix is $\tau^r(S')\in \GL_d(L^{\per}(t))\cap \Mat_{d\times d}(L(t))=\GL_d(L(t))$.
\item By the proof of the first part, we see that one only has to invert the numerator of $\det(\tau^r(S'))=\det(\tau^r(S))\in L(t)$, as well as all (finitely many) denominators occuring in $D'$, resp.~the product $f$ of all these, to obtain $\tau^r\kappa_j$ ($j=1,\ldots,d$) as $L[t][{f}^{-1}]$-linear combination of the $\tau^i\kappa_l$ ($l=1,\ldots,d$,$0\leq i<r$). The higher $\tau$-images of the $\kappa_j$ are obtained if we also invert the higher $\tau$-twists of $f$.
\item 
First take $f$ and $F$ as in \ref{item:2}.
As $\mot_{\locLA}$ is finitely generated, and $\locLA$ is a Noetherian ring, also the torsion submodule $T$ is finitely generated. In particular, there are generators $m_1,\ldots, m_k\in T$ of $T$, and $f_1,\ldots, f_k\in \LA$ such that $f_i\cdot m_i=0$ for all $i$.
Replacing $f$ by $f\cdot f_1\cdots f_k$, the torsion submodule $T$ is killed in the corresponding new localization $\mot_{F^{-1}(\LA)}$.
Hence, the new localization $\mot_{\locLA}$ is torsion-free. As $\locLA$ is even a Dedekind domain, the torsion-free module  $\mot_{\locLA}$ is locally free. Hence, there exist $g\in \LA$ such that $\mot_{\locLAGtilde}$ is a free $\locLAGtilde$-module, where $\tilde{G}$ is the multiplicative subset generated by $F$ and $g$. Therefore, taking $G$ to be the multiplicative subset generated by $g\cdot f$ and all its $\tau$-twists, $\locLAG$ is a localization of $\locLAGtilde$, and 
 also $\mot_{\locLAG}$ is a free $\locLAG$-module.
\end{enumerate}
\end{proof}

\begin{exmp}\label{exmp:t-module-with-t-theta-torsion}
Quasi-periodic extensions of uniformizable abelian $t$-modules as defined in \cite[Section 3.3]{db-mp:ligvpc}, are examples of uniformizable non-abelian $t$-modules where one can take $f=t-\ell(t)$. In particular, their $t$-motives are free at all primes $\pfrak\lhd A$.
\end{exmp}

\begin{exmp}\label{exmp:t-module-with-t-division}
Let $(E,\phi)$ be the two-dimensional $t$-module over $\Fq(\theta)$ given by
\[ \phi_t=\begin{pmatrix} \theta & 0 \\ \theta & \theta \end{pmatrix}+
\begin{pmatrix} 1 & 1 \\ 0 &0\end{pmatrix}\cdot \tau, \]
where $\theta=\ell(t)$.
So, we have to diagonalize the matrix 
\[\tau^{-1}(D-t\cdot \one_2)=\begin{pmatrix} \tau^{-1}\cdot (\theta-t)+1 & 1 \\ \tau^{-1}\cdot \theta & \tau^{-1}\cdot (\theta-t)
\end{pmatrix}\in \Mat_{2\times 2}(\Fq(\theta)(t)\{\tau^{-1}\})\]
by elementary row and column operations. As the upper right entry is $1$, this is quite easy and one obtains for example the matrix
\[\begin{pmatrix} 1 & 0 \\ 
0 & -\tau^{-2}\cdot (\theta^q-t)(\theta-t)+\tau^{-1}t \end{pmatrix}\]
So choosing $T=\left( \begin{smallmatrix}1 & 0 \\ 0 & \tau \end{smallmatrix}\right)$, one gets
\[ T\cdot \begin{pmatrix} 1 & 0 \\ 
0 & -\tau^{-2}\cdot (\theta^q-t)(\theta-t)+\tau^{-1}t \end{pmatrix} \equiv  \begin{pmatrix} 1 & 0 \\ 
0 & t \end{pmatrix}=:S \mod \tau^{-1}.\]

So the theorem asserts that the $t$-motive of $E$ becomes finitely generated after inverting $\det(\tau^r(S))=t$.
Pursuing the calculations of the proof further, one obtains the relation
\[ \begin{pmatrix} -\tau\cdot (\theta^q-t)(\theta-t)+\tau^2\cdot t & 0 \\ 
(\theta^q-t)^2(\theta-t)+\tau\cdot (\theta-t)(1-2t+\theta^q)+\tau^2\cdot (1-t) & 1 
\end{pmatrix} \begin{pmatrix} \kappa_1 \\ \kappa_2\end{pmatrix} =  \begin{pmatrix} 0 \\ 0\end{pmatrix}.
\]
As $\tau$ acts torsionfree, the first equation can be divided from the left by $\tau$, so we end up with the relations:
\begin{eqnarray*}
t\cdot \tau\kappa_1 &=& (\theta^q-t)(\theta-t)\kappa_1 \\
\kappa_2 &=& -\left( (\theta^q-t)^2(\theta-t)+\tau\cdot (\theta-t)(1-2t+\theta^q)+\tau^2\cdot (1-t)\right)\kappa_1
\end{eqnarray*}
The second equation could be further reduced using the first, but we already see that $\mot$ is not finitely generated as $\Fq(\theta)[t]$-module, but we have infinite $t$-division. We also see that $\Fq(\theta)[t,\frac{1}{t}]\otimes_{\Fq(\theta)[t]} \mot(E)$ is free as $\Fq(\theta)[t,\frac{1}{t}]$-module of rank $1$ with basis $\{\kappa_1\}$.\\
This also shows that $\rk(\mot)=1$.
\end{exmp}

\medskip

A corresponding statement as the one in Theorem \ref{thm:rational-A-motive-f-dim} holds for the dual $A$-motive $\dumot=\Hom_{\vs}(\GG_{a,L},E)$ as well.\\
Although we don't deal with the dual $A$-motive in the next sections, we state the theorem here for sake of completeness.
As the dual $A$-motive is usually only defined for perfect fields $L$, we restrict to this case.
Since, the places of $E$ and $\GG_{a,L}$ are switched compared to the motive $\mot$, the dual motive $\dumot$ is not a left-$L\{\tau\}$-module, but a right-$L\{\tau\}$-module, and usually considered as a left module for the opposite ring $L\{\sigma\}=L\{\tau\}^*$ where
$\sigma=\tau^{-1}$ (see e.g.~\cite{db-mp:ridmtt}).

\begin{thm}\label{thm:rational-dual-A-motive-f-dim}
Let $L$ be perfect. Denote by $\Quot(\LA)$ the field of fractions of $\LA$.
The dual $A$-motive $\dumot:=\dumot(E)$ associated to $E$ has the following properties as $\LA$-module:
\begin{enumerate}
\item $\Quot(\LA)\otimes_{\LA}\dumot $ is a finite dimensional $\Quot(\LA)$-vector space. 
\item There is an element $f\in \LA$ such that one has:
Let $F$ be the multiplicative subset generated by $f$ and all its images under (iterated) $\sigma$-twists, then $\locLA\otimes_{\LA} \dumot$ is finitely generated as $\locLA$-module.
\item The element $f\in \LA$ and the corresponding set $F$ as in \eqref{item:2}, can be chosen such that $\locLA\otimes_{\LA} \dumot$ is a free module.
\end{enumerate}
\end{thm}

\begin{proof}[Sketch of proof]
We consider the dual $A$-motive not as a left module over the opposite ring of $L\{\tau\}$ (as it is usually done), but as a right-$L\{\tau\}$-module. Then with respect to some $L\{\tau\}$-basis $\dual{\kappa}_1,\ldots,\dual{\kappa}_d$ of $\dumot$, we can write
\[\begin{pmatrix} \dual{\kappa}_1& \dual{\kappa}_2 & \cdots & \dual{\kappa}_d\end{pmatrix}\cdot t = \begin{pmatrix} \dual{\kappa}_1& \dual{\kappa}_2 & \cdots & \dual{\kappa}_d\end{pmatrix}\cdot D\]
for some $D\in \Mat_{d\times d}(L\{\tau\})$. Hence, the matrix $D-t\cdot \one_d\in L[t]\{\tau\}$ annihilates the $L\{\tau\}$-basis $\dual{\kappa}_1,\ldots,\dual{\kappa}_d$.
The first part is then obtained in the same way as for Theorem \ref{thm:rational-A-motive-f-dim}\eqref{item:1}, but with sides swapped, and one gets that for some $r\geq 0$ all $\dual{\kappa}_j\tau^r$ ($j=1,\ldots, d$) are $L[t][f^{-1}]$-linear combinations of the $\dual{\kappa}_l\tau^i$ ($l=1,\ldots,d$,$0\leq i<r$).\footnote{More details about the swap from left to right can be found in \cite[Sect.~6]{am:aefam}.} The higher $\tau$-images of the $\dual{\kappa}_j$, i.e.~the $\dual{\kappa}_j\tau^i$ with $i>r$ are obtained if we also invert the higher $\tau^{-1}$-twists of $f$.\\
The third part is deduced from the second part in the same way as in the previous theorem.
\end{proof}

\subsection{Equality of ranks}\label{subsec:equality-of-ranks}

The goal of this subsection is to show the following theorem.
\begin{thm}\label{thm:equality-of-ranks}
Let $E$ be an Anderson $A$-module over $L$ and $\mot$ its $A$-motive. Then for almost all primes $\pfrak$ of $A$, we have
\[ \trk{\pfrak}(E)=\rk(\mot). \]
If $E$ is uniformizable, this equality holds for all primes $\pfrak$ of $A$.
\end{thm}

Recall from Definition \ref{def:localization}, that for a prime ideal $\pfrak$ of $A$, we set
\[  (\LA)_{(\pfrak)}:= \bigcap_{\qfrak \mid (\LA)\pfrak} (\LA)_{(\qfrak)} \subseteq \Quot(\LA)\]
to be the intersection of the localizations of $\LA$ at all prime ideals $\qfrak$ dividing the extended ideal $(\LA)\pfrak$.
As there are only finitely many such prime ideals $\qfrak$, this ring is a semi-local ring.

\begin{defn}
For a prime $\pfrak$ of $A$, the $A$-motive $\mot=\mot(E)$ is called \emph{free at $\pfrak$}, if the localization
\[ \mot_{(\pfrak)}:=(\LA)_{(\pfrak)}\otimes_{\LA} \mot\]
is a free and finitely generated $(\LA)_{(\pfrak)}$-module.
\end{defn}

\begin{rem}
Given an abelian Anderson module $E$, its $A$-motive $\mot=\mot(E)$ is a finitely generated locally free $\LA$-module. Hence, for all primes $\pfrak$ of $A$, the localization $\mot_{(\pfrak)}$ is locally free. Since $(\LA)_{(\pfrak)}$ is semi-local, this implies that $\mot_{(\pfrak)}$ is free. Therefore, for $A$-motives of abelian Anderson modules,
the condition of being free at $\pfrak$ is satisfied for all primes $\pfrak$ of $A$.
\end{rem}

For non-abelian Anderson modules, we have the following.

\begin{prop}\label{prop:free-almost-everywhere}
Let $E$ be a (non-abelian) Anderson module. For almost all primes $\pfrak$ of $A$, the $A$-motive $\mot(E)$  is free at $\pfrak$.
\end{prop}

\begin{proof}
Theorem \ref{thm:rational-A-motive-f-dim}\eqref{item:3} guarantees that the $A$-motive is free at all primes $\pfrak$ of $A$ whose intersection with the constructed set $F$ is empty. As the primes $\pfrak$ are stable under the $\tau$-twist, these are exactly those primes $\pfrak$ such that the extended ideal $(\LA)\pfrak\subseteq \LA$ does not contain the element $f$. Since, $f$ is contained in only finitely many primes, we conclude from Theorem \ref{thm:rational-A-motive-f-dim} that $\mot_{(\pfrak)}$ is free for all but finitely many primes $\pfrak$ of $A$.
\end{proof}

\begin{exmp}\label{exmp:t-module-with-t-division-part2}
Let's consider again the $t$-module of Example \ref{exmp:t-module-with-t-division}. We have seen that $\mot$ has infinite $t$-division, and that $\Fq(\theta)[t,\frac{1}{t}]\otimes_{\Fq(\theta)[t]} \mot(E)$ is free as $\Fq(\theta)[t,\frac{1}{t}]$-module of rank $1$.
Hence, $\mot$ is free at all primes $\pfrak$ different from $(t)$, but not at the prime $(t)$.
\end{exmp}

The next goal is the following proposition.

\begin{prop}\label{prop:local-rank}
If $\mot=\mot(E)$ is free at $\pfrak$, one has
\[ \rk_{(\LA)_{(\pfrak)}}(\mot_{(\pfrak)}) = \trk{\pfrak}(E). \]
\end{prop}

As the rank of a free module is stable under base extension,
we will work throughout the rest of this section with
$\tilde{\mot}:= (\CA)\otimes_{\LA} \mot$ instead of $\mot$.
Be aware that with this notation,
$\tilde{\mot}_{(\pfrak)}=(\CA)_{(\pfrak)}\otimes_{(\LA)_{(\pfrak)}} \mot_{(\pfrak)}$, hence, if  $\mot$ is free at $\pfrak$, also $\tilde{\mot}$ is free at $\pfrak$.

\begin{lem}\label{lem:iso-mod-pfrak}
Let $\pfrak\lhd A$ be a prime ideal such that $\tilde{\mot}$ is free at $\pfrak$.
\begin{enumerate}
\item \label{item:bar-tau-iso} The map induced by $\tau$,
\[ \bar{\tau}: \tilde{\mot}/\pfrak\tilde{\mot} \to \tilde{\mot}/\pfrak\tilde{\mot}, \overline{m}\mapsto \overline{\tau(m)} \]
is a semi-linear isomorphism of finite dimensional $\CI$-vector spaces.
\item $\tilde{\mot}/\pfrak\tilde{\mot}$ has a basis consisting of elements in $(\tilde{\mot}/\pfrak\tilde{\mot})^\tau$.
\item \label{item:padic-invariants} Let $\tilde{\mot}_\pfrak=\varprojlim_{n} \left( \tilde{\mot}/\pfrak^{n+1}\tilde{\mot}\right)$ be the $\pfrak$-adic completion with the induced $(\CA)_\pfrak$-module structure and the induced $\tau$-action. 
Then $\tilde{\mot}_\pfrak$ has a basis consisting of elements in $(\tilde{\mot}_\pfrak)^\tau$.
\end{enumerate} 
\end{lem}

\begin{proof}
Since $\tilde{\mot}$ is free at $\pfrak$, and so $\tilde{\mot}_{(\pfrak)}/\pfrak \tilde{\mot}_{(\pfrak)}$ is a free $(\CA)_{(\pfrak)}$-module of finite rank, and since $\tilde{\mot}/\pfrak\tilde{\mot}\cong \tilde{\mot}_{(\pfrak)}/\pfrak \tilde{\mot}_{(\pfrak)}$,
the quotient $\tilde{\mot}/\pfrak\tilde{\mot}$ is finitely generated as a module over $(\CA)/(\pfrak)\cong (A/\pfrak)\otimes \CI\cong \Fpp\otimes \CI$. Therefore it is a finite dimensional $\CI$-vector space. So for showing part \ref{item:bar-tau-iso}, it suffices to show that $\bar{\tau}$ is surjective.

Let $J\lhd \CA$ be the ideal generated by all elements $a\otimes 1-1\otimes \ell(a)$ for $a\in A$. By definition of the $A$-action on $E$ and on $\tilde{\mot}$, we have $J^d\tilde{\mot} \subseteq \tau(\tilde{\mot})$. Further the ideal $(\pfrak)$ is relatively prime to $J^d$, and hence there exist $f\in \pfrak$ and $x\in J^d$ such that $1=f+x$.

Therefore, for all $m\in \tilde{\mot}$, we have $m=fm+xm\in \pfrak\tilde{\mot} + \tau(\tilde{\mot})\subseteq \tilde{\mot}$. In other words, every element $\overline{m}\in \tilde{\mot}/\pfrak\tilde{\mot}$ is in the image of $\bar{\tau}$.

The second part is obtained by Lang isogeny (cf. first part of \cite[Lemma 1.8.2]{ga:tm} for more details).

For proving the third part, let $\vect{m}=\transp{(m_1,\ldots, m_r)}$ be an $(\CA)_\pfrak$-basis of $\tilde{\mot}_\pfrak$ which is a lift of a basis of elements in $(\tilde{\mot}/\pfrak\tilde{\mot})^\tau$, and let $\Theta \in \Mat_r((\CA)_\pfrak)$ be such that
\[ \tau( \vect{m}) = \Theta\vect{m}. \]
By choice of $\vect{m}$, we have $\Theta\equiv \one_r \mod{\pfrak}$, and in particular $\Theta\in \GL_r((\CA)_\pfrak)$. We will show by induction that there is a sequence $(\vect{m}_{n})_{n\in\NN}$ of bases of $\tilde{\mot}_{\pfrak}$ such that for all $n\geq 0$, 
\[ \tau(\vect{m}_n)\equiv \vect{m}_n \mod{\pfrak^{n+1}}\tilde{\mot}\]
and for all $n\geq 1$,
\[ \vect{m}_n\equiv \vect{m}_{n-1} \mod{\pfrak^{n}}\tilde{\mot}. \]
The limit of this sequence is then a basis of $\tilde{\mot}_\pfrak$ consisting of $\tau$-invariant elements.

The case $n=0$ is given by $\vect{m}_0:=\vect{m}$. For $n\geq 1$, let $\Theta_{n-1}$ be such that $\tau(\vect{m}_{n-1})=\Theta_{n-1} \vect{m}_{n-1}$. The induction hypothesis on $\vect{m}_{n-1}$ ensures that
\[ \Theta_{n-1}\equiv \one_r \mod{\pfrak^n}. \]

As $\CI$ is separably algebraically closed, the map $x\mapsto \tau(x)-x$ is surjective on $\Fpp\otimes \CI$. Hence there exists a matrix $Y\in \pfrak^{n}\Mat_r((\CA)_\pfrak)$ such that
\[  \tau(Y)-Y \equiv \Theta_{n-1}-\one_r \mod{\pfrak^{n+1}}. \]
Setting $\vect{m}_n:=(\one_r-Y)\cdot \vect{m}_{n-1}\equiv \vect{m}_{n-1}\mod{\pfrak^{n}}\tilde{\mot}$, we therefore obtain
\begin{eqnarray*}
\tau(\vect{m}_n) &=& \tau\left((\one_r-Y)\cdot \vect{m}_{n-1}\right)\\
&=& (\one_r-\tau(Y))\cdot \Theta_{n-1}\cdot \vect{m}_{n-1}\\
&\equiv & \left( (\one_r-\tau(Y))+(\one_r-\tau(Y))(\tau(Y)-Y) \right) \cdot \vect{m}_{n-1} \mod{\pfrak^{n+1}}\tilde{\mot}\\
&\equiv & \left( \one_r-\tau(Y) + \tau(Y)-Y \right) \cdot \vect{m}_{n-1} \mod{\pfrak^{n+1}}\tilde{\mot}\\
&= & \vect{m}_{n}
\end{eqnarray*} 
\end{proof}

\begin{cor}\label{cor:mots-dualizable}
Let $\pfrak\lhd A$ be a prime ideal such that $\mot$ is free at $\pfrak$. 
\begin{enumerate}
\item $\tilde{\mot}/\pfrak\tilde{\mot}$, $\tilde{\mot}_{(\pfrak)}$ and $\tilde{\mot}_{\pfrak}$ are dualizable difference modules over
$\Fpp\otimes\CI$, over $(\CA)_{(\pfrak)}$ and over $(\CA)_{\pfrak}$, respectively.
\item $\mot/\pfrak\mot$, $\mot_{(\pfrak)}$ and $\mot_{\pfrak}$ are dualizable difference modules over
$\Fpp\otimes L$, over $(\LA)_{(\pfrak)}$ and over $(\LA)_{\pfrak}$, respectively.
\end{enumerate}
\end{cor}

\begin{proof}
By definition, all six modules are free of finite rank over the respective difference rings. So we only have to check that they satisfy one (and hence all) of the equivalent conditions of Lemma \ref{lem:N-dualizable}.

We start with the first part.
By Lemma \ref{lem:iso-mod-pfrak}, $\tilde{\mot}/\pfrak\tilde{\mot}$ has a basis $\bar{\vect{m}}=\transp{(\bar{m}_1,\ldots, \bar{m}_r)}$ of $\tau$-invariant elements. So this basis satisfies $\tau(\bar{\vect{m}})=\one_r\bar{\vect{m}}$, and hence condition \eqref{item:Phi-invertible} of Lemma \ref{lem:N-dualizable} is satisfied for $\tilde{\mot}/\pfrak\tilde{\mot}$.
We have already seen in the proof of the third part of Lemma \ref{lem:N-dualizable}, if we lift this basis $\bar{\vect{m}}$ to a basis $\vect{m}=\transp{(m_1,\ldots, m_r)}$ of $\tilde{\mot}_{\pfrak}$, the matrix $\Theta$ defined by $\tau( \vect{m}) = \Theta\vect{m}$ is invertible, as it is congruent to $\one_r$ modulo $\pfrak$. So also $\tilde{\mot}_{\pfrak}$ is a dualizable difference module.
The very same argument shows that also $\tilde{\mot}_{(\pfrak)}$ satisfies condition \eqref{item:Phi-invertible} of Lemma \ref{lem:N-dualizable}.

For the second part, we have to recognize that an element $\alpha$ in $\Fpp\otimes L$, in $(\LA)_{(\pfrak)}$ or in $(\LA)_{\pfrak}$ is invertible, if it is invertible as an element of $\Fpp\otimes\CI$, of $(\CA)_{(\pfrak)}$ and of $(\CA)_{\pfrak}$, respectively.
Then condition \ref{item:Phi-invertible} of Lemma \ref{lem:N-dualizable} is easily verified for each module using the result for the corresponding module of the first part.
\end{proof}

\begin{proof}[Proof of Proposition \ref{prop:local-rank}]

Using Corollary \ref{cor:mots-dualizable} and Lemma \ref{lem:iso-mod-pfrak}\eqref{item:padic-invariants} in combination with Lemma \ref{lem:generated-by-invariants}, as well as $A_\pfrak = (\CA)_\pfrak^{\,\tau}$, we conclude that we have an isomorphism of $(\CA)_\pfrak$-modules
\begin{equation}\label{eq:padic-iso}
(\CA)_\pfrak\otimes_{A_\pfrak}  (\tilde{\mot}_\pfrak)^\tau \cong \tilde{\mot}_\pfrak,
\end{equation}
and that
\begin{equation}\label{eq:padic-equal-rank}
\rk_{A_\pfrak}\left( (\tilde{\mot}_\pfrak)^\tau\right) = \rk_{(\CA)_\pfrak} \left( \tilde{\mot}_\pfrak\right).
\end{equation}

Another incredient of the proof is the adjointness of scalar restriction and scalar extension, i.e.~that for commutative rings $R\subseteq S$, an $R$-module $M$, and an $S$-module $N$, we have a natural isomorphism
\begin{equation}\label{eq:restr-ext}
\Hom_S(S\otimes_R M,N) \cong \Hom_R(M,N).
\end{equation}

We obtain the following chain of isomorphisms:
\begin{eqnarray*}
\Hom_{\LA}(\mot, (\CA)_\pfrak) &\stackrel{\eqref{eq:restr-ext}}{\cong} & \Hom_{\CA}(\tilde{\mot},(\CA)_\pfrak)\\
&\stackrel{\eqref{eq:restr-ext}}{\cong} & \Hom_{(\CA)_\pfrak}\left(\tilde{\mot}_\pfrak,(\CA)_\pfrak\right) \\
&\stackrel{\eqref{eq:padic-iso}}{\cong} & \Hom_{(\CA)_\pfrak}\left((\CA)_\pfrak\otimes_{A_\pfrak}  (\tilde{\mot}_\pfrak)^\tau,(\CA)_\pfrak\right)\\
&\stackrel{\eqref{eq:restr-ext}}{\cong} & \Hom_{A_\pfrak}\left((\tilde{\mot}_\pfrak)^\tau,(\CA)_\pfrak\right).
\end{eqnarray*}

As all isomorphisms are compatible with the $\tau$-action, we can restrict to $\tau$-equivariant homomorphisms.
\begin{equation*}
\hat{H}_\pfrak \stackrel{\ip}{\cong}  \Hom_{\LA}^\tau(\mot, (\CA)_\pfrak) 
\cong \Hom_{A_\pfrak}^\tau\left((\tilde{\mot}_\pfrak)^\tau,(\CA)_\pfrak\right)
\cong \Hom_{A_\pfrak}\left((\tilde{\mot}_\pfrak)^\tau,A_\pfrak\right).
\end{equation*}
Hence,
\begin{eqnarray*}
\trk{\pfrak}(E) &\stackrel{\text{Rem.}\ref{rem:p-rk-rk-Hp}}{=}& \rk_{A_\pfrak}(\hat{H}_\pfrak)=\rk_{A_\pfrak}\left(\Hom_{A_\pfrak}((\tilde{\mot}_\pfrak)^\tau,A_\pfrak)\right) \\
&=& \rk_{A_\pfrak}( (\tilde{\mot}_\pfrak)^\tau) \stackrel{\eqref{eq:padic-equal-rank}}{=}\rk_{(\CA)_\pfrak} (\tilde{\mot}_\pfrak) 
=\rk_{(\LA)_{(\pfrak)}} (\mot_{(\pfrak)}).
\end{eqnarray*}
\end{proof}

\begin{proof}[Proof of Theorem \ref{thm:equality-of-ranks}]
If $\mot$ is free at $\pfrak$, then one has
\[ \rk(\mot)=\dim_{\Quot(\LA)} \left(\Quot(\LA)\otimes_{\LA} \mot\right)  = \rk_{(\LA)_{(\pfrak)}} \left((\LA)_{(\pfrak)}\otimes_{\LA} \mot \right), \]
and further by Proposition \ref{prop:local-rank},
\[ \rk_{(\LA)_{(\pfrak)}} \left((\LA)_{(\pfrak)}\otimes_{\LA} \mot \right)=\trk{\pfrak}(E). \]
By Proposition \ref{prop:free-almost-everywhere}, the motive $\mot$ is free at almost all primes $\pfrak$. Hence, we conclude the first part.
For uniformizable Anderson modules, we have additionally that $\trk{\pfrak}(E)$ does not depend on $\pfrak$ by Corollary \ref{cor:uniformizable-implies-constant-prank}. Hence, the equality has to hold not only for almost all primes, but for all primes.
\end{proof}

\begin{rem}
Our proof also recovers Anderson's statement that for abelian Anderson modules the equality of the ranks holds for all primes. Namely in that case, $\mot$ is a locally free $\LA$-module, and hence $\mot$ is free at all primes $\pfrak$.
\end{rem}

\section{Uniformizability and rigid analytic trivializations}\label{sec:rigid-analytic-trivializations}

Recall that an Anderson $A$-module $E$ is called \emph{uniformizable}, if its exponential map $\e:\LieE\to E(\CI)$ is surjective.

In the case $A=\Fq[t]$, Anderson proved that for an abelian Anderson $A$-module $E$ the following are equivalent (see \cite[Theorem 4]{ga:tm}):
\begin{enumerate}
\item $E$ is uniformizable,
\item its period lattice $\Lambda_E$ is a free $A$-module of rank $r=\rk(E):=\rk(\mot)$,
\item its motive $\mot$ is rigid analytically trivial.
\end{enumerate}

Here, the $A$-motive $\mot$ of an abelian Anderson module $E$ is defined to be \emph{rigid analytically trivial} if the natural homomorphism
$\TT \otimes_A (\mot_\TT)^\tau \to \mot_\TT$ is an isomorphism (see \cite[Definition 3.18]{uh-akj:pthshcff}), where $\mot_\TT:=\TT\otimes_{\LA}\mot$.

Since in the abelian case, $\mot$ is a locally free $\LA$-module, and hence $\mot_\TT$ is a locally free $\TT$-module,
 $\mot$ is rigid analytically trivial if and only if the dual 
\[\dual{(\mot_\TT)} = \Hom_{\TT}(\mot_\TT,\TT)=\Hom_{\LA}(\mot,\TT)\isom \TT\otimes_{\LA} \dual{\mot}\] 
is rigid analytically trivial, if and only if
 $\TT\otimes_A \left(\dual{(\mot_\TT)}\right)^\tau \to \dual{(\mot_\TT)}$ is an isomorphism.

In the non-abelian case, we prove the following chain of implications that is close to that equivalence of conditions.

\begin{thm}\label{thm:conditions-for-uniformizability}
Let $E$ be an Anderson $A$-module over $L$, and $r=\rk(\mot)$ the rank of its $A$-motive.
Consider the statements:
\begin{enumerate}
\item \label{item:cond-rat-and-loc-free} $\mot_\TT$ is locally free and $\TT\otimes_A \left(\dual{(\mot_\TT)}\right)^\tau \to \dual{(\mot_\TT)}$ is an isomorphism.
\item \label{item:cond-E-unif} $E$ is uniformizable.
\item \label{item:cond-lattice-of-full-rank} $\Lambda_E$ is an $A$-module of rank $r$.
\item \label{item:cond-rat} $\dual{(\mot_\TT)}$ is locally free of rank $r$ and $\TT\otimes_A \left(\dual{(\mot_\TT)}\right)^\tau \to \dual{(\mot_\TT)}$ is an isomorphism.
\end{enumerate}
The following three implications hold:
\[ \eqref{item:cond-rat-and-loc-free} \Rightarrow \eqref{item:cond-E-unif} \Rightarrow \eqref{item:cond-lattice-of-full-rank} \Rightarrow \eqref{item:cond-rat}. \]
\end{thm}

\subsection{Preparations from commutative algebra}

Let $\cR$ be a Dedekind domain, and $\cF=\Quot(\cR)$ its field of fractions.

\begin{prop}\label{prop:dual-locally-free}
 Let $M$ be an $\cR$-module such that $\cF\otimes_\cR M$ is a finite dimensional vector space.
Then $\dual{M}:=\Hom_\cR(M,\cR)$ is a finitely generated locally free $\cR$-module.
\end{prop}

\begin{proof}
As $\cR$ is torsionfree as an $\cR$-module, we have
\[ \Hom_\cR(M,\cR) = \Hom_\cR(M/M^{tor},\cR), \]
where $M^{tor}$ is the torsion submodule of $M$. Hence, without loss of generality, we can assume that $M$ is torsionfree.
Then, the natural map $M\to \cF\otimes_\cR M$ is injective.

We have a natural embedding 
\begin{equation}\label{eq:embedding-of-M-dual}
 \dual{M}=\Hom_\cR(M,\cR)\subset \Hom_\cR(M,\cF) = \Hom_\cF(\cF\otimes_\cR M,\cF),
\end{equation}
where the second equality stems from equation \eqref{eq:restr-ext}. Hence, $\dual{M}$ is torsionfree, too.

Taking $\mu_1,\ldots, \mu_s\in \Hom_\cR(M,\cR)$ a maximal $\cF$-linear independent subset%
, there exist $m_1,\ldots, m_s\in M\subset \cF\otimes_\cR M$ such that
$\mu_i(m_j)=0$ for all $i\ne j$ and $\mu_i(m_i)\ne 0$ for all $i$.
Further for any $\mu\in \Hom_\cR(M,\cR)$, there are $x_i\in \cF$ such that $\mu=\sum_{i=1}^s x_i\mu_i$.
Then for all $i=1,\ldots s$,
\[  x_i\mu_i(m_i)=\mu(m_i)\in \cR, \]
i.e.~ $x_i\in \mu_i(m_i)^{-1}\cR\subset \cF$.
Therefore, $\Hom_\cR(M,\cR)$ is an $\cR$-submodule of the $\cR$-module generated by $\mu_i(m_i)^{-1}\mu_i\in \Hom_\cR(M,\cF)$ for $i=1,\ldots, s$.
As $\cR$ is Noetherian, this implies that $\Hom_\cR(M,\cR)$ is also finitely generated, and as $\cR$ is even a Dedekind domain, every finitely generated torsionfree $\cR$-module is locally free.
\end{proof}

\subsection{Proof of Theorem \ref{thm:conditions-for-uniformizability}}

The proof of Theorem \ref{thm:conditions-for-uniformizability} takes the whole of this subsection.

We start with the implication \eqref{item:cond-E-unif}$\Rightarrow$\eqref{item:cond-lattice-of-full-rank}, i.e.~if $E$ is uniformizable, then $\rk_A(\Lambda_E)=r=\rk(\mot)$. 
\begin{proof}[Proof of \eqref{item:cond-E-unif}$\Rightarrow$\eqref{item:cond-lattice-of-full-rank}]
Since by \cite[Theorem 3.11]{qg-am:sfgtshrd}, we have an isomorphism $\fd_{A/\Fq[u]}\cdot \Lambda_E \xrightarrow{\delta_u} \fsf(E)$ (see Section \ref{sec:homomorphisms}), and the different ideal $\fd_{A/\Fq[u]}$ is a rank one $A$-module, we have $\rk_A(\Lambda_E)=\rk_A(\fsf(E))$.
By Corollary \ref{cor:uniformizable-implies-constant-prank}, we have $\rk_A(\fsf(E))=\trk{\pfrak}(E)$ for all primes $\pfrak$ of $A$, and
by Theorem \ref{thm:equality-of-ranks}, we have $\rk(\mot)=\trk{\pfrak}(E)$ for almost all primes $p$.
In total, we get
\[ \rk_A(\Lambda)=\rk(\mot)=r. \qedhere \]
\end{proof}

We proceed with the implication \eqref{item:cond-lattice-of-full-rank}$\Rightarrow$\eqref{item:cond-rat}, i.e.~if $\rk_A(\Lambda)=r$, then $\dual{(\mot_\TT)}$ is locally free of rank $r$ and $\TT\otimes_A \left(\dual{(\mot_\TT)}\right)^\tau \to \dual{(\mot_\TT)}$ is an isomorphism.
\begin{proof}[Proof of \eqref{item:cond-lattice-of-full-rank}$\Rightarrow$\eqref{item:cond-rat}]
Combining the isomorphisms $\iota$ and $\delta_u$, we obtain an isomorphism
\[ \iota\circ \delta_u: \fd_{A/\Fq[u]}\cdot \Lambda_E \to \Hom_{\LA}^\tau(\mot,\TT),\]
and
\[ \Hom_{\LA}^\tau(\mot,\TT)=\Hom_{\TT}^\tau(\mot_\TT,\TT)=\left(\dual{(\mot_\TT)}\right)^\tau. \]
Hence, $\left(\dual{(\mot_\TT)}\right)^\tau$ is a locally free $A$-module of rank $r$.

By Theorem \ref{thm:rational-A-motive-f-dim}, $\mot_\TT$ satisfies the hypothesis of Proposition \ref{prop:dual-locally-free}, and
therefore, $\dual{(\mot_\TT)}$ is locally free.
Since, the homomorphism $\TT\otimes_A \left(\dual{(\mot_\TT)}\right)^\tau \to \dual{(\mot_\TT)}$ is always injective by \cite[Lemma 4.2(b)]{gb-uh:uft}, and $\dual{(\mot_\TT)}\subseteq \LL\otimes_\TT \dual{(\mot_\TT)}=\dual{(\mot_\LL)}$, we have
\[ r= \rk_A(\left(\dual{(\mot_\TT)}\right)^\tau \leq \rk_\TT\left(\dual{(\mot_\TT)}\right) \leq \dim_\LL(\dual{(\mot_\LL)})=\dim_\LL(\mot_\LL)=r,\]
i.e.~$\rk_\TT\left(\dual{(\mot_\TT)}\right) =r$.
We conclude by \cite[Lemma 4.2(c)]{gb-uh:uft} which states that $\rk_A(\left(\dual{(\mot_\TT)}\right)^\tau = \rk_\TT\left(\dual{(\mot_\TT)}\right)$ implies that the given homomorphism is an isomorphism.
\end{proof}

The implication \eqref{item:cond-rat-and-loc-free}$\Rightarrow$\eqref{item:cond-E-unif} is proved in the same way as Anderson's theorem in the abelian $t$-motive case (see \cite[Section 2.6]{ga:tm}). We recall the main steps here, and start with an observation.
\begin{lem}\label{lem:Delta-surjective}
The $A$-linear map $\Delta:\TT\to \TT, x\mapsto (1-\tau)(x)=x-\tau(x)$ is surjective with kernel $A=\TT^\tau$.
\end{lem}

\begin{proof}
$\Ker(\Delta)=\TT^\tau=A$ is clear. Every $y\in \TT$ can be written as a convergent series $y=\sum_{i \geq 0} a_i\otimes d_i$.
Since $\CI$ is algebraically closed, each equation $X-X^q-d_i=0$ has $q$ solutions in $\CI$, and for all $d_i\in \CI$ with $\norm{d_i}<1$, one solution $c_i\in \CI$ satisfies $\norm{c_i}= \norm{d_i}$.
Choosing these solutions (and any solution when $\norm{d_i}\geq 1$), we have $x=\sum_{i\geq 0}a_i\otimes c_i\in \TT$, and we obtain
\[ \Delta(x)=\sum_{i\geq 0}a_i\otimes c_i-\sum_{i\geq 0}a_i\otimes c_i^q = \sum_{i\geq 0}a_i\otimes (c_i-c_i^q)
=\sum_{i\geq 0}a_i\otimes d_i=y.\]
Hence, $\Delta$ is surjective.
\end{proof}

\begin{proof}[Proof of \eqref{item:cond-rat-and-loc-free}$\Rightarrow$\eqref{item:cond-E-unif}]
Let $K_\infty$ be the completion of $K$ at the infinite place. The split exact sequence of $A$-modules
\[0\to A \to K_\infty\to K_\infty/A\to 0\]
(which is split exact as a sequence of $\Fq$-vector spaces) induces a short exact sequence of $(\CA)\{\tau\}$-modules
\[  0\to \Hom^c(K_\infty/A, \CI)\to \Hom^c(K_\infty, \CI)\to \Hom^c(A, \CI) \to 0. \]
Here $\Hom^c(-,\CI)$ denotes the space of continuous homomorphisms of $\Fq$-vector spaces,

In the following, we abbreviate $R=(\CA)\{\tau\}$.

By applying the functor $\Hom_{R}(\mot, -)$ to that sequence, we obtain a left exact sequence
\begin{equation}\label{eq:left-exact-sequence}
0 \to \Hom_{R}(\mot, \Hom^c(K_\infty/A, \CI) )\to \Hom_{R}(\mot, \Hom^c(K_\infty, \CI))\to \Hom_{R}(\mot, \Hom^c(A, \CI))
\end{equation}
This sequence is indeed isomorphic to the sequence $0\to \Lambda_E\to \LieE\xrightarrow{\e} E(\CI)$ in the sense that we have a commutative diagram
\[\xymatrix{
0 \ar[r] & \Lambda_E \ar[r]\ar[d]^{f_1|_{\Lambda_E}} & \LieE \ar[r]^{\e} \ar[d]^{f_1} & E(\CI) \ar[d]^{f_2}\\
0 \ar[r] & \Hom_{R}(\mot, \Hom^c(K_\infty/A, \CI) )\ar[r] & \Hom_{R}(\mot, \Hom^c(K_\infty, \CI))\ar[r]  & \Hom_{R}(\mot, \Hom^c(A, \CI)) 
}\]
where the vertical arrows are isomorphisms, namely given by
\[ f_1(\lambda)(m)(c) = m(\e(\dphi_c \lambda))\quad \forall c\in K_\infty, m\in \mot,\lambda\in \LieE \]
and
\[ f_2(e)(m)(a) = m(\phi_a e)\qquad \forall a\in A,m\in \mot, e\in E(\CI). \]
Commutativity of the diagram is easily verified. The map $f_2$ is an isomorphism, because it is the composition of the isomorphisms
$E(\CI)\to \Hom_{L\{\tau\}}(\mot,\CI)$ (used already in the proof of Theorem \ref{thm:h-isom-M-tate-dual}) and the isomorphism
$\Hom_{L\{\tau\}}(\mot,\CI)\to \Hom_{R}(\mot, \Hom^c(A, \CI))$ sending a homomorphism $g$ to the homomorphism $f_g$ defined by
\[  f_g(m)(a)=g(am) \quad \forall a\in A,m\in \mot. \]

For verifying that $f_1|_{\Lambda_E}$ is an isomorphism and -- by exactness of the rows -- also $f_1$ is an isomorphism, we first remark that
the map
\[ \alpha: \Omega_{A/\Fq}\otimes_A \TT \to \Hom^c(K_\infty/A, \CI), \omega\otimes (\sum_i a_i\otimes c_i)\mapsto f_\omega \]
where
\[ f_\omega(\bar{x})=\sum_i \tr_{\FI/\Fq}(\res_\infty(a_ix\omega))\cdot c_i\]
is an isomorphism of $R=\LA\{\tau\}$-modules. Therefore, $f_1|_{\Lambda_E}$ factors as
\[ \Lambda_E\to \Hom_{R}(\mot, \Omega_{A/\Fq}\otimes_A \TT)\to \Hom_{R}(\mot, \Hom^c(K_\infty/A, \CI)),\]
and it is not hard to see that the first map is obtained from the isomorphism
$\iota\circ \delta_u:\fd_{A/\Fq[u]}\cdot \Lambda_E \to \Hom_{\LA}^\tau(\mot,\TT)$ by tensoring with $\Omega_{A/\Fq}$.

\medskip

For showing that $\e$ is surjective, we have to show that the sequence \eqref{eq:left-exact-sequence} is also right exact. By cohomology theory, the sequence \eqref{eq:left-exact-sequence} continues to a long exact sequence with the terms
\[ \ldots \to \Ext^1_{R}(\mot, \Hom^c(K_\infty/A, \CI))\xrightarrow{\gamma} \Ext^1_{R}(\mot, \Hom^c(K_\infty, \CI))\to \ldots \]
We are done, if we can show that $\gamma$ is injective. In particular, this is true, if we can show that \linebreak $\Ext^1_{R}(\mot, \Hom^c(K_\infty/A, \CI))=0$.

\medskip

One has 
\begin{eqnarray*}
 \Ext^1_{R}(\mot, \Hom^c(K_\infty/A, \CI))&\cong &\Ext^1_{R}(\mot, \Omega_{A/\Fq}\otimes_A \TT)\\
 &\cong &
\Coker\left(\Hom_R(F_0,\Omega_{A/\Fq}\otimes_A \TT)\to \Hom_R(F_1,\Omega_{A/\Fq}\otimes_A \TT)\right),
\end{eqnarray*}
where $0\to F_1\to F_0\to \mot\to 0$ is any $R$-resolution of $\mot$ with $F_0$ a projective $R$-module.
As $\Omega_{A/\Fq}\otimes_A \TT$ is even a module over $Q=\TT\{\tau\}\cong \TT\otimes_{\LA} R$, we have
\begin{eqnarray*}
\Ext^1_{R}(\mot, \Omega_{A/\Fq}\otimes_A \TT)&\cong &
\Coker\left(\Hom_R(F_0,\Omega_{A/\Fq}\otimes_A \TT)\to \Hom_R(F_1,\Omega_{A/\Fq}\otimes_A \TT)\right)\\
&\cong & \Coker\left(\Hom_Q(Q\otimes_R F_0,\Omega_{A/\Fq}\otimes_A \TT)\to \Hom_Q(Q\otimes_R F_1,\Omega_{A/\Fq}\otimes_A \TT)\right)\\
&\cong & \Ext^1_{Q}(\mot_\TT, \Omega_{A/\Fq}\otimes_A \TT),
\end{eqnarray*}
as
\[ 0\to Q\otimes_R F_1\to Q \otimes_R F_0\to \mot_{\TT}\to 0\]
is a $Q$-resolution of the $Q$-module $\mot_{\TT}$ with $Q\otimes_R F_0$ a projective $Q$-module.

However, $\Ext^1_Q$ does not depend on the chosen $Q$-resolution of $\mot_\TT$.

By Lemma \ref{lem:generated-by-invariants}, hypothesis \eqref{item:cond-rat-and-loc-free} implies that also
$\mot_\TT\cong \TT\otimes_A (\mot_\TT)^\tau$, and $(\mot_\TT)^\tau$ is a projective $A$-module.
So, another projective $Q$-resolution is given by
\[ 0\to Q\otimes_A (\mot_\TT)^\tau \xrightarrow{\alpha} Q\otimes_A (\mot_\TT)^\tau \to \TT\otimes_A (\mot_\TT)^\tau\cong \mot_\TT\to 0,\]
where the first map $\alpha$ is given by $q\otimes m\mapsto q(1-\tau)\otimes m$, and the second map is given by
$q\otimes m\mapsto q|_{\tau=1}\otimes m= q\cdot m \in \mot_\TT$.

Hence,
\begin{eqnarray*}
 &&\Ext^1_{Q}(\mot_\TT, \Omega_{A/\Fq}\otimes_A \TT)\\ &&\quad \cong \Coker\left(\Hom_Q(Q\otimes_A (\mot_\TT)^\tau,\Omega_{A/\Fq}\otimes_A \TT)\xrightarrow{\alpha^*} \Hom_Q(Q\otimes_A (\mot_\TT)^\tau,\Omega_{A/\Fq}\otimes_A \TT) \right)\\
 &&\quad \cong \Coker\left(\Hom_A((\mot_\TT)^\tau,\Omega_{A/\Fq}\otimes_A \TT)\xrightarrow{\beta} \Hom_A( (\mot_\TT)^\tau,\Omega_{A/\Fq}\otimes_A \TT) \right)
\end{eqnarray*}
where
\[ \beta(f)(m)=(1-\tau)f(m) \quad \forall f\in \Hom_A((\mot_\TT)^\tau,\Omega_{A/\Fq}\otimes_A \TT), m\in (\mot_\TT)^\tau. \]
However, by Lemma \ref{lem:Delta-surjective}, the map $\Delta:\TT\to \TT, x\mapsto \Delta(x):=(1-\tau)(x)=x-\tau(x)$ is surjective, and $(\mot_\TT)^\tau$ is a projective $A$-module, and hence $\beta$ is also surjective, i.e. the cokernel of $\beta$ is $0$.
\end{proof}

\begin{rem}
We do not know whether all four conditions in Theorem \ref{thm:conditions-for-uniformizability} are equivalent.
However, the gap is not big.

The condition that $\dual{(\mot_\TT)}$ has maximal rank $r=\dim_\LL(\mot_\LL)$ implies that 
$\mot_\TT/\mot_\TT^{tor}$ is finitely generated. Since $\TT$ is a Dedekind domain, we conclude that $\mot_\TT/\mot_\TT^{tor}$ is locally free.

So the difference between condition \eqref{item:cond-rat-and-loc-free} and \eqref{item:cond-rat} is that for the latter, $\mot_\TT$ might have torsion.

Along the lines of the proof above, the critical point in showing \eqref{item:cond-rat}$\Rightarrow$\eqref{item:cond-E-unif} is to show that
\[  \Ext^1_{R}(\mot, \Hom^c(K_\infty/A, \CI))\xrightarrow{\gamma} \Ext^1_{R}(\mot, \Hom^c(K_\infty, \CI)) \]
is injective, even when $\mot_\TT$ has torsion.

On the other hand, if one could show that this map is not injective, whenever $\mot_\TT$ has torsion, this would show the equivalence of \eqref{item:cond-rat-and-loc-free} and \eqref{item:cond-E-unif}.
\end{rem}

\begin{exmp}\label{exmp:t-module-with-t-division-part3}
We consider the $t$-module from Example \ref{exmp:t-module-with-t-division}.
We have seen that $\rk(\mot)=1$, and that $\mot$ is torsion free, but it has infinite $t$-division. So, also $\mot_\TT$ has infinite $t$-division, and hence $\dual{(\mot_\TT)}=\Hom_{\TT}(\mot_\TT,\TT)=0$.
By Theorem \ref{thm:conditions-for-uniformizability}, we conclude that $E$ is not uniformizable.
\end{exmp}

\section{Rigid analytic trivializations and Galois representations}

In this section, we introduce rigid analytic trivializations of $A$-motives associated to (possibly non-abelian) Anderson $A$-modules.
This notion will be a generalization of the rigid analytic trivializations in the abelian case where $\mot$ is a free $\LA$-module.

These rigid analytic trivializations are used to give a description of the  $\pfrak$-adic Galois representations. 

\subsection{Rigid analytic trivializations}

Since in the non-abelian case, the motive is not finitely generated, we have to consider localizations thereof. Such localizations are also necessary in the abelian case, when the motive is not free, but only locally free.

\begin{defn}
Let $F\subset \LA$ be a $\tau$-stable multiplicative subset such that $\mot[F^{-1}]=(\LA[F^{-1}])\otimes_{\LA} \mot$ is a free $\LA[F^{-1}]$-module of rank $r=\rk(\mot)$.\footnote{Theorem \ref{thm:rational-A-motive-f-dim} guarantees that such a multiplicative subset exists.}
A \emph{rigid analytic trivialization} for $\mot[F^{-1}]$ with respect to some $\LA[F^{-1}]$-basis $\transp{(m_1,\ldots,m_r)}$ of 
$\mot[F^{-1}]$ is a matrix $\Upsilon\in \GL_r(\TT[F^{-1}])$ satisfying
\[ \tau\left(\Upsilon\cdot  \svect{m}{r}\right) = \Upsilon\cdot  \svect{m}{r}.\]
\end{defn}

\begin{rem}\label{rem:rigid-analytic-trivialization}
If $\Upsilon\in \GL_r(\TT[F^{-1}])$ is a rigid analytic trivialization with respect to $\transp{(m_1,\ldots,m_r)}$, the basis
$\Upsilon\cdot \svect{m}{r}$ is a basis of $\tau$-invariant elements of $\mot_{\TT[F^{-1}]}=\TT[F^{-1}]\otimes_{(\LA[F^{-1}])}\mot[F^{-1}]$ by definition.
On the other hand, if $\mot_{\TT[F^{-1}]}$ has a basis of $\tau$-invariant elements, the base change matrix from that basis to $\transp{(m_1,\ldots,m_r)}$ is a rigid analytic trivialization. Hence, the existence of a rigid analytic trivialization $\Upsilon\in \GL_r(\TT[F^{-1}])$ is equivalent to the condition that the $\TT[F^{-1}]$-module $\mot_{\TT[F^{-1}]}$ has a basis of $\tau$-invariant elements. 
\end{rem}

\begin{thm}\label{thm:rigid-analytic-trivialization}
Let $F\subset \LA$ be a $\tau$-stable multiplicative subset such that $\mot[F^{-1}]=(\LA[F^{-1}])\otimes_{\LA} \mot$ is a free $\LA[F^{-1}]$-module of rank $r=\rk(\mot)$.
If $E$ is uniformizable, there exists a rigid analytic trivialization for $\mot[F^{-1}]$.
\end{thm}

\begin{proof}
By Theorem \ref{thm:conditions-for-uniformizability}, if $E$ is uniformizable, the space $\dual{(\mot_\TT)}$ has a basis of $\tau$-invariant elements.
Hence, the same holds for $\TT[F^{-1}]\otimes_\TT \dual{(\mot_\TT)} \cong \dual{(\mot_{\TT[F^{-1}]})}$. By Lemma \ref{lem:generated-by-invariants}, this implies that $\mot_{\TT[F^{-1}]}$ also has a basis of $\tau$-invariant elements.
\end{proof}

\begin{cor}\label{cor:existence-of-local-rat}
Let $E$ be uniformizable, and $\pfrak\subseteq A$ a prime ideal such that $\mot_{(\pfrak)}$ is a free $(\LA)_{(\pfrak)}$-module of rank $r=\rk(\mot)$. Then there exists a rigid analytic trivialization for $\mot_{(\pfrak)}$.
\end{cor}

\begin{proof}
Since, $(\pfrak)$ is a $\tau$-stable ideal, $(\LA)_{(\pfrak)}$ is a localization of the form $\LA[F^{-1}]$ with a $\tau$-stable multiplicative subset $F$. So the statement directly follows from the previous theorem.
\end{proof}

\subsection{Galois representations}

It is well known that all torsion points of $E$ have coefficients in the separable closure $L^\sep$ of $L$ inside $\CI$ (cf.~Remark \ref{rem:torsion-in-Lsep}). As the action of the Galois group $\Gal(L^\sep/L)$ commutes with the action of $\phi$, every automorphism $\gamma\in \Gal(L^\sep/L)$ induces automorphisms on the Tate-modules.

This leads to the well-known Galois representations
\[ \varrho_\pfrak: \Gal(L^\sep/L)\to \Aut(T_\pfrak(E))\cong \GL_r(A_\pfrak), \]
 for any prime $\pfrak\subset A$ for which $\trk{\pfrak}(E)=r$, where the isomorphism $\Aut(T_\pfrak(E))\cong \GL_r(A_\pfrak)$ is given via some choice of $A_\pfrak$-basis of $T_\pfrak(E)$.

\medskip

In this section, we give an explicit description of these Galois representations in terms of a rigid analytic trivialization of its associated $A$-motive $\mot=\mot(E)$.

In the following, we use the natural extension of the canonical action of  $\Gal(L^\sep/L)$ on $L^\sep$  to actions on $(A\otimes E(L^\sep))_\pfrak$ and on $\Hom_{\LA}^\tau(\mot, (A\otimes L^\sep)_\pfrak)$. Namely, we extend the action to $A\otimes L^\sep$ and $A\otimes E(L^\sep)$ by acting on the second factor, and then continuously to $(A\otimes L^\sep)_\pfrak$, and to $(A\otimes E(L^\sep))_\pfrak$. The Galois action on $\Hom_{\LA}^\tau(\mot, (A\otimes L^\sep)_\pfrak)$ is then induced by the action on the codomain $(A\otimes L^\sep)_\pfrak$.

\begin{lem}\label{lem:compatibilities}
The following hold:
\begin{enumerate}
\item $\hat{H}_\pfrak\subset (A\otimes E(L^\sep))_\pfrak$, and $\hat{H}_\pfrak$ is stable under the action of  $\Gal(L^\sep/L)$,
\item $\Hom_{\LA}^\tau(\mot, (\CA)_\pfrak)= \Hom_{\LA}^\tau(\mot, (A\otimes L^\sep)_\pfrak)$,
\item the isomorphisms $\jp$ and $\ip$ are compatible with the action of $\Gal(L^\sep/L)$.
\end{enumerate}
\end{lem}

\begin{proof}
Since, $T_\pfrak(E)=\Hom_A(K_\pfrak/A_\pfrak, E(\CI))=\Hom_A(K_\pfrak/A_\pfrak, E(L^\sep))$, 
the first point is clear from the explicit isomorphism $\jp$ given in Theorem \ref{thm:isom-jp}. The second point is then obtained from the explicit isomorphism $\ip$, and the compatibility with the action of $\Gal(L^\sep/L)$ is also easily verified using these explicit isomorphisms.
\end{proof}

\begin{thm}\label{thm:galois-representation}
Let $\pfrak$ be a prime ideal of $A$ such that $\mot$ is free at $\pfrak$. Assume that there exists a rigid analytic trivialization $\Upsilon\in \GL_r(\TT_{(\pfrak)})$  of $\mot_{(\pfrak)}$ with respect to some $(\LA)_{(\pfrak)}$-basis $\transp{(m_1,\ldots,m_r)}$ of $\mot_{(\pfrak)}$.\footnote{By Corollary \ref{cor:existence-of-local-rat}, this holds for example if $E$ is uniformizable.} 

Using the embedding $\TT_{(\pfrak)}\to (\CA)_\pfrak$, we consider $\Upsilon\in \GL_r((\CA)_\pfrak)$.
\begin{enumerate}
\item The coefficients of $\Upsilon$ lie in $(\LsepA)_\pfrak$.
\item \label{item:explicit-representation} With respect to an appropriate choice of basis of $T_{\pfrak}(E)$,
the $\pfrak$-adic Galois representation $\varrho_\pfrak$ attached to $E$ is given by
\begin{eqnarray*}
\varrho_\pfrak:\Gal(L^{\sep}/L) &\to& \GL_r(A_\pfrak) \\
\gamma &\mapsto & \Upsilon\cdot \gamma(\Upsilon)^{-1}.
\end{eqnarray*} 

\end{enumerate}
\end{thm}

\begin{rem}
\begin{enumerate}
\item The formula implicitly states that for all $\gamma\in \Gal(L^{\sep}/L)$, the product $\Upsilon\cdot \gamma(\Upsilon)^{-1}\in \GL_r((\CA)_\pfrak)$ has coefficients in $A_\pfrak$ which is the ring of invariants of $(\CA)_\pfrak$. Using that $\Upsilon$ satisfies the difference equation
\[ \tau(\Upsilon)=\Upsilon \Phi^{-1}, \]
where $\Phi\in \GL_{r}((\LA)_{(\pfrak)})$ is given by $\tau(\vect{m})=\Phi \vect{m}$, we can verify this by a small computation:
\begin{eqnarray*}
\tau(\Upsilon\cdot \gamma(\Upsilon)^{-1}) &=& \tau(\Upsilon)\cdot \tau(\gamma(\Upsilon))^{-1} 
= \tau(\Upsilon)\cdot \gamma(\tau(\Upsilon))^{-1}, \quad \text{since $\gamma$ commutes with $\tau$, }\\
&=& \Upsilon \Phi^{-1}\cdot \gamma\left( \Upsilon \Phi^{-1}\right)^{-1} 
= \Upsilon \Phi^{-1} \cdot \gamma(\Phi)\cdot \gamma(\Upsilon)^{-1} \\
&=& \Upsilon\cdot \gamma(\Upsilon)^{-1}, \qquad\quad \text{since $\Phi$ has coefficients in $(\LA)_{(\pfrak)}\subseteq (\LA)_\pfrak$.} 
\end{eqnarray*} 
\item We obtain the formulas in \cite[Proposition~5.1]{am-rp:tcagfdte} by evaluating further at a prime in $\CA$ dividing $\pfrak$.
Be aware that the notation there uses the same symbol but for slightly different things: $\Upsilon$ here is $(\transp{\Upsilon})^{-1}$ there.
\end{enumerate}
\end{rem}

\begin{proof}
By Lemma \ref{lem:compatibilities}, the isomorphisms $\jp$ and $\ip$ are compatible with the action of $\Gal(L^\sep/L)$. Hence, the action on $T_\pfrak(E)$ and the action on $\Hom_{\LA}^\tau(\mot, (A\otimes L^\sep)_\pfrak)$ provide the same Galois representations with respect to corresponding bases.

Let $(\mu_1,\ldots, \mu_r)$ be the basis of $\Hom_{\LA}(\mot_{(\pfrak)},(\LA)_{(\pfrak)})$ dual to $\transp{(m_1,\ldots,m_r)}$. Then
$(\mu_1,\ldots, \mu_r)\cdot \Upsilon^{-1}$ is an $A_{(\pfrak)}$-basis of $\Hom_{\LA}^\tau(\mot, \TT_{(\pfrak)})$, and hence an $A_\pfrak$-basis of $\Hom_{\LA}^\tau(\mot, (A\otimes \CI)_\pfrak)=\Hom_{\LA}^\tau(\mot, (A\otimes L^\sep)_\pfrak)$. In particular, the coefficients of $\Upsilon$ lie in $(A\otimes L^\sep)_\pfrak$.

Finally, the image of the basis $(\mu_1,\ldots, \mu_r)\cdot \Upsilon^{-1}$ under $\gamma\in \Gal(L^\sep/L)$ is
\begin{eqnarray*} \gamma\bigl((\mu_1,\ldots, \mu_r)\cdot \Upsilon^{-1}\bigr) &=& (\mu_1,\ldots, \mu_r)\cdot \gamma(\Upsilon)^{-1} \\
&=& (\mu_1,\ldots, \mu_r)\cdot \Upsilon^{-1}\cdot \left( \Upsilon \cdot \gamma(\Upsilon)^{-1}\right) .
\end{eqnarray*}
This shows the claim.
\end{proof}

\begin{rem}
Assume that $F\subseteq \LA$ is a multiplicative subset such that $\mot[F^{-1}]$ is a free $\LA[F^{-1}]$-module of rank $r=\rk(\mot)$, and that there exists a rigid analytic trivialization for $\mot[F^{-1}]$. Then this rigid analytic trivialization can be used in the previous theorem for every prime $\pfrak\subset A$ for which $F$ has empty intersection with $(\pfrak)$.
In particular in the abelian case, if $\mot$ is free (e.g.~for $A=\Fq[t]$), a rigid analytic trivialization $\Upsilon\in \GL_r(\TT)$ can be used for all primes $\pfrak$.
\end{rem}

\begin{rem}
In Remark \ref{rem:up-expansion-by-Taylor-series}, we have seen that the embedding $\TT_{(\pfrak)}\to (\CA)_\pfrak$ can be explicitly given via the Taylor series expansion at $\pfrak$ by replacing $X$ by the uniformizer $\up{}$, i.e.
\[h \mapsto \sum_{n\geq 0} \up{n}\cdot\hd{\up{}}{n}{h}(\pfrak).\]
\end{rem}

\subsection{Image of the Galois representation as subgroup of the motivic Galois group}

We end this paper, by the generalization of \cite[Theorem 5.2]{am-rp:tcagfdte}. That means, we use the explicit description of the Galois representations in the previous subsection to directly show that for a uniformizable Anderson module $E$, the image of the $\pfrak$-adic Galois representation lies in the motivic Galois group, i.e.~in the Tannakian Galois group $\Gamma_{\mot}$ of the rational $A$-motive 
$\mot_Q=Q\otimes_{\LA} \mot(E)$ where $Q=\Quot(\LA)$ is the field of fraction of $\LA$.

For achieving this, we use the same characterization of the Galois group $\Gamma_\mot$ as in \cite[Sect.~4.2]{mp:tdadmaicl}, but for the $A$-motive and not the dual $A$-motive.
More recent results in difference Galois theory (e.g.~in the work of Bachmayr \cite[Sect.2]{am:dvnt}), and for more general transcendental Galois theories -- so called  Picard-Vessiot theories -- (see \cite{am:capvt}) provide the necessary theoretical background. 

\medskip

So let $E$ be a uniformizable Anderson module, and let $\pfrak$ be a prime of $A$ such that $\mot$ is free at $\pfrak$.
Let $\vect{m}=\transp{(m_1,\ldots,m_r)}$ be a basis of $\mot_{(\pfrak)}$. Let $\Theta\in \Mat_{r\times r}(\LA)$ be the matrix such that
\[ \tau(\vect{m})=\Theta \vect{m}. \]
By Corollary \ref{cor:mots-dualizable}, we have $\Theta\in \GL_r(\LA)$.

Recall that by Corollary \ref{cor:existence-of-local-rat}, there exists a rigid analytic trivialization $\Upsilon\in \GL_r(\TT_{(\pfrak)})$ with respect to the basis $\vect{m}=\transp{(m_1,\ldots,m_r)}$, i.e.~satisfying $\tau(\Upsilon\vect{m})=\Upsilon\vect{m}$.
This amounts to
\[ \tau(\Upsilon)=\Upsilon\Theta^{-1}, \]
and we let $\Psi:=\Upsilon^{-1}$ to get a difference equation in the form that is standard in difference Galois theory (with the matrix $\Psi$ as right factor on the right hand side):
\begin{equation}\label{eq:difference-equation-for-Psi}
\tau(\Psi)=\Theta \Psi. 
\end{equation} 
Associated to the matrix $\Psi\in \GL_r(\TT_{(\pfrak)})$ is a difference Galois group $\GPsi$, and \cite[Corollary 7.8]{am:capvt} states (in a quite general setting) that the motivic Galois group $\Gamma_{\mot}$ is isomorphic to that difference Galois group $\GPsi$.

\medskip

So we recall the difference Galois group $\GPsi$ associated to the matrix $\Psi\in \GL_r(\TT_{(\pfrak)})$ (see also the appendix of \cite{qg-am:ctmcrzvp}).

As above, let $Q=\Quot(\LA)$ be the field of fractions of $\LA$, and let $\LL=\Quot(\TT)$ be the field of fractions of $\TT$, and recall that $K=\Quot(A)$ denotes the field of fraction of $A$.
We denote by $P=Q[\Psi,\det(\Psi)^{-1}]\subset \LL$ the so-called Picard-Vessiot ring for the equation~\eqref{eq:difference-equation-for-Psi}.

The difference Galois group $\GPsi$ is defined as the group functor
\begin{align*}
\GPsi: ~& \CAlg_K& \longrightarrow & \quad \CGrp , \\
& ~\sA & \longmapsto & \quad \Gamma_{\Psi}(\sA):= \left\{
\alpha\in \Aut_{Q\otimes_K \sA}(P\otimes_K \sA) \,\, \middle|\,\, 
(\tau\otimes \id)\circ \alpha =\alpha\circ (\tau\otimes \id) \right\} &
\end{align*}
Since any element $\alpha\in \GPsi(\sA)$ commutes with the $\tau$-action in the given manner, the solution matrix $\Psi\in \GL_r(P)\subseteq \GL_r(P\otimes_K\sA)$ is mapped to another solution matrix which has to be of the form $\alpha(\Psi)=\Psi C(\alpha)$ for a matrix
$C(\alpha)\in \GL_r(\sA)\subseteq \GL_r(P\otimes_K \sA)$.

This provides the embedding of $\GPsi$ as a subgroup of $\GL_{r,K}$, i.e.~for all $K$-algebras $\sA$:
\[ \GPsi(\sA) \to \GL_r(\sA), \alpha \mapsto C(\alpha)=\Psi^{-1}\cdot \alpha(\Psi). \]

For determining which elements of $\GL_r(\sA)$ belong to (the image of) $\GPsi(\sA)$, consider the localized polynomial ring 
$Q[Y_{ij},\det(Y)^{-1}| i,j=1,\ldots,r ]$, and the ideal
\[ I:= \left\{ h\in Q[Y_{ij},\det(Y)^{-1}] \,\middle|\, h(\Psi)=0 \right\}, \]
so that $P$ is isomorphic to the quotient $Q[Y_{ij},\det(Y)^{-1}]/I$ by mapping the entries of the matrix $Y$ to the one of $\Psi$.
Then for $C\in \GL_r(\sA)$, we have
\begin{equation}\label{eq:condition-for-GPsi}
C=C(\alpha) \text{ for some }\alpha\in \GPsi(\sA) \quad \Longleftrightarrow \quad \forall\, f\in I:  f(\Psi\otimes C)=0\in P\otimes_K \sA.\footnote{For matrices $M$ and $N$, the notation $M\otimes N$ is shorthand for the matrix
$\left(\sum_{l} M_{il}\otimes N_{lj}\right)_{ij}$ which equals the product $\left( M_{ij}\otimes 1\right)_{ij}\cdot \left( 1\otimes N_{ij}\right)_{ij}$.} 
\end{equation} 

\begin{thm}\label{thm:image-in-motivic-group}
Let $E$ be a uniformizable Anderson module, and $\pfrak$ a prime of $A$ such that $\mot$ is free at $\pfrak$.
With notation as above, we have:
The explicit $\pfrak$-adic Galois representation from Theorem \ref{thm:galois-representation}\eqref{item:explicit-representation},
$
\varrho_\pfrak:\Gal(L^{\sep}/L) \to \GL_r(A_\pfrak), 
\gamma \mapsto  \Upsilon\cdot \gamma(\Upsilon)^{-1}$,
has image in
$\GPsi(K_\pfrak)$. 
\end{thm}

\begin{proof}
By the previous discussion, we just have to show that for every $\gamma\in \Gal(L^{\sep}/L)$, the matrix $C=\Upsilon\cdot \gamma(\Upsilon)^{-1}$ satisfies the condition in \eqref{eq:condition-for-GPsi} for $\sA=K_\pfrak$.
So let $\gamma\in \Gal(L^{\sep}/L)$ and $f\in I$ be arbitrary. As $C=\Upsilon\cdot \gamma(\Upsilon)^{-1}\in \GL_r(A_\pfrak)$ actually has coefficients in $(\LsepA)_\pfrak$, and also the coefficients of $\Psi=\Upsilon^{-1}$ are in $(\LsepA)_\pfrak$, the tensor
$\Psi\otimes C$ is just the product $\Psi\cdot C\in \GL_r((\LsepA)_\pfrak)$.
Then
\begin{eqnarray*}
f(\Psi\otimes \left( \Upsilon\cdot \gamma(\Upsilon)^{-1}\right) ) &=&  f(\Psi\cdot  \Upsilon\cdot \gamma(\Upsilon)^{-1}) \\
&=& f\left(\Psi\cdot \Psi^{-1}\cdot \gamma(\Psi)\right) = f(\gamma(\Psi))\\
&=& \gamma \left( f(\Psi) \right)= 0.
\end{eqnarray*}
The second last equation holds, since $f$ has coefficients in $Q=\Quot(\LA)$ which is fixed by $\gamma$.
\end{proof}

\bibliographystyle{alpha}
\def\cprime{$'$}

\vspace*{.5cm}

\parindent0cm

\end{document}